\documentclass[reqno]{amsart}
\usepackage{amssymb}
\usepackage[colorlinks=true]{hyperref}
\usepackage[margin=1in]{geometry}
\usepackage{mathabx}
\usepackage{enumitem}
\usepackage{bbm}

\def\Xint#1{\mathchoice
   {\XXint\displaystyle\textstyle{#1}}%
   {\XXint\textstyle\scriptstyle{#1}}%
   {\XXint\scriptstyle\scriptscriptstyle{#1}}%
   {\XXint\scriptscriptstyle\scriptscriptstyle{#1}}%
   \!\int}
\def\XXint#1#2#3{{\setbox0=\hbox{$#1{#2#3}{\int}$}
     \vcenter{\hbox{$#2#3$}}\kern-.5\wd0}}
\def\dashint{\Xint-}

\newtheorem{theorem}{Theorem}[section]

\newtheorem{lemma}[theorem]{Lemma}
\newtheorem{proposition}[theorem]{Proposition}

\theoremstyle{remark}
\newtheorem{observation}[theorem]{Remark}
\newtheorem{assumption}{Assumption}

\newtheorem{definition}{Definition}[section]

\newcommand{\les}{\lesssim}
\newcommand{\dd}{\,d}
\newcommand{\R}{\mathbb R}
\newcommand{\C}{\mathbb C}
\newcommand{\Z}{\mathbb Z}
\newcommand{\lb}{\label}
\newcommand{\be}{\begin{equation}}
\newcommand{\ee}{\end{equation}}
\newcommand{\B}{\mathcal B}
\newcommand{\mc}{\mathcal}
\newcommand{\K}{\mc K}

\newcommand{\ov}{\overline}

\DeclareMathOperator{\supp}{supp}

\DeclareMathOperator{\sgn}{sgn}

\DeclareMathOperator{\one}{\mathbbm{1}}
\DeclareMathOperator{\hh}{\mc H}
\renewcommand{\H}{\hh}
\newcommand{\D}{\mc D}
\newcommand{\M}{\mc M}
\DeclareMathOperator{\esssup}{ess sup}

\newenvironment{itm}
{\begin{itemize}[leftmargin=0pt]
  \setlength{\itemsep}{0pt}
  \setlength{\parskip}{0pt}
  \setlength{\parsep}{0pt}}
{\end{itemize}}

\title[Spectral multipliers III]{Spectral multipliers III: Endpoint bounds, intertwining operators, and twisted Hardy spaces}
\author{Marius Beceanu}
\author{Michael Goldberg}

\begin{document}

\begin{abstract} We extend several fundamental estimates regarding spectral multipliers for the free Laplacian on $\R^3$ to the case of perturbed Hamiltonians of the form $H=-\Delta+V$, where $V$ is a scalar real-valued potential.

Results include sharp bounds for Mihlin multipliers, partial confirmation for a conjecture made in \cite{becgol3} about intertwining operators, a characterization of the twisted Hardy spaces that correspond to these perturbed Hamiltonians, Strichartz estimates, and maximum principles.
\end{abstract}
\maketitle

\tableofcontents
\section{Introduction}

In this paper we establish certain fundamental properties of Schr\"odinger operators $H = -\Delta + V$, with potentials $V$ of finite global Kato norm
\be\lb{katonorm}
\|V\|_\K := \sup_{y \in \R^3} \int_{\R^3} \frac{|V(x)| \dd x}{|x-y|}.
\ee
Results include sharp bounds for Mihlin multipliers, proving some results conjectured in \cite{becgol3} about intertwining operators, upgrading previous Strichartz estimates from \cite{becgol2} and \cite{becgol3}, and maximum principles.

\subsection{Main result}
We assume that $V$ belongs to the closure of $\D=C^\infty_c(\R^3)$ in the global Kato class $\K=\{V \in \M \mid \|V\|_\K < \infty\}$, which we denote $\K_0$. This ensures that $V \in L^1_{loc}$ and that $H=-\Delta+V$ is self-adjoint and bounded from below~\cite{simon}.

By results in \cite{becgol4}, it follows that $\sigma_p(H)$ is finite and consists only of negative eigenvalues, possibly together with $0$, which can be an eigenvalue or a resonance or both. The condition that $V \in L^1_{loc}$ can sometimes be dispensed with, see \cite{becgol4}.
If $V$ is large, non-positive eigenvalues can be present.

Our goal is to extend certain estimates from the free Laplacian to the perturbed case. However, many of them do not hold for the subspace spanned by the bound states, in the perturbed case. The behavior associated to this subspace can be determined easily, since it is finite dimensional and spanned by smooth and rapidly decaying functions. Still, the analogous estimates to the ones in the free case apply only to the (absolutely) continuous part of the spectrum, that is only after we project away the bound states.

Henceforth we shall make the following spectral assumption:
\begin{assumption} $H=-\Delta+V$ has no zero energy eigenstates or resonances.
\end{assumption}
The non-generic (see \cite{becgol4}) cases when threshold eigenstates or resonances are present lead to different estimates for the (absolutely) continuous spectrum, even after projecting these bound states away, and need to be considered separately; see, for example, \cite{erdschlag} or \cite{bec3}.

Propagators for the perturbed wave equation and radially symmetric Mihlin multipliers are both spectral multipliers, being of the form $f(H)$ for some appropriate function $f$. Since $H$ is self-adjoint here, $f(H)$ can be defined in terms of the spectral measure
\be \lb{functionalcalc}
f(H) = \int_{\sigma(H)} f(\lambda) \dd E_H(\lambda)
\ee
as a bounded operator on $L^2$ whenever $f$ is a bounded Borel measurable function on $\sigma(H)$.

The classical H\"{o}rmander--Mihlin theorem establishes that if $m$ satisfies the Mihlin (\ref{cm}) or H\"{o}rmander (\ref{ch}) conditions, then $m(-\Delta)$ is bounded from $L^1$ to $L^{1, \infty}$ and on the Hardy space $\H$, on $L^p$ for $1<p<\infty$, and on $BMO$.

In this paper, we prove a similar result for $m(H)$. We are able to handle the endpoints $L^1$, $\H$, and $L^\infty$ with no extra assumptions on $H$. This is a sharp result, stronger than the one in \cite{becgol4}.

There are several such theorems currently in the literature, each having different assumptions about the operator $H$ or the multiplier function $m$, but none being able to handle the endpoints. For more details, we refer the reader to \cite{becgol3}.

\begin{theorem}[Main Theorem]\lb{mainthm}
Assume $V \in \mc K_0$, and $H = -\Delta + V$ has no eigenvalue or resonance at
zero energy, and no positive eigenvalues.  Let $\phi$ be a standard cutoff function such that $\phi \in C^\infty$, $\supp \phi \subset [\frac12, 4]$, and $\phi(x)=1$ for $x \in [1, 2]$.

Suppose $m: (0,\infty) \to \mathbb C$ satisfies H\"{o}rmander's condition
\begin{equation} \label{ch}
\sup_{k\in \Z} \|\phi(\lambda)m(2^{-k}\lambda)\|_{H^s} = M_s < \infty \text{ for some } s> \frac d2 = \frac32
\end{equation}
or equivalently $\sup_{\alpha>0} \|\phi(\lambda)m(\alpha^{-1} \lambda)\|_{H^s} < \infty$.
Then $m(\sqrt H)=m(\sqrt H) P_c$ is a singular integral operator of weak-(1, 1) type and $L^p$-bounded for $1<p<\infty$. Moreover, $m(\sqrt{H})(x, y) \in L^\infty_y L^{1, \infty}_x$.

Furthermore, if in addition $V \in L^{3/2, \infty}$ and $m$ satisfies Mihlin's stronger condition
\be\lb{cm}
|m^{(k)}(\lambda)| \leq C_2 \lambda^{-k}
\ee
for $0 \leq k \leq \lfloor\frac d 2\rfloor+1 = 2$, then $m(\sqrt H)H(-\Delta)^{-1}$ is bounded from the Hardy space $\H$ to $L^1$.

All norms are bounded by constants that depend linearly on $M_s$ or $C_2$.
\end{theorem}

\begin{observation}
The conclusions of Theorem~\ref{mainthm} hold equally well for $m(H)$, because $m(\lambda)$ satisfies the H\"{o}rmander condition (\ref{ch}) if and only if $m(\lambda^2)$ does.
\end{observation}

The $L^p$ boundedness for $1<p<\infty$ admits a simplified proof, see \cite{becgol3}, based on decomposing the kernel into the free part and the difference, assuming the classical result that the free part $m(-\Delta)$ is bounded and also that $V \in L^{3/2, \infty}$. Under the stronger assumption that $s>2$, \cite{becgol3} proved a pointwise $|x-y|^{-3}$ kernel bound. Under a further assumption on $V$, the same paper proved the weak-type (1, 1) boundedness of spectral Mihlin multipliers.

\subsection{The twisted Hardy space}
In Theorem \ref{mainthm} we saw that, due to its cancellation properties, $m(H) H(-\Delta)^{-1}$ is bounded from the Hardy space $\H$ to $L^1$. By duality, $(-\Delta)^{-1} H m(H)$ is also bounded from $L^\infty$ to $BMO$.

Naturally, we want such a result for $m(H)$ itself. For this purpose, we introduce the twisted Hardy space, corresponding to the perturbed Hamiltonian $H=-\Delta+V$ in the same manner in which the Hardy space $\H$ corresponds to the free Hamiltonian $-\Delta$.

The simplest definition would be taking the twisted Hardy space to be
$$
HP_c(-\Delta)^{-1} \H = P_c(I+V(-\Delta)^{-1})\H.
$$
However, this has the disadvantage that Mihlin multipliers $m(H)$, such as $H^{i\sigma}$, are not necessarily bounded from this space to itself, so we shall use a more complicated, but more general definition:

\begin{definition}\lb{dhardy} The twisted Hardy space is
$$
\tilde \H = \{f \in L^1 \mid f=\int_\R H^{1+i\sigma}P_c(-\Delta)^{-1} f_\sigma \dd \sigma,\ \int_\R \|f_\sigma\|_{\H} \langle \sigma\rangle^2 \dd \sigma < \infty\},
$$
and with the $\tilde \H$ norm taken as the infimum over such decompositions.

Let $\tilde {BMO}$ be the dual of $\tilde \H$.
\end{definition}
Note that the space spanned by the bound states of $H$, if any, is not included. Thus, $\tilde \H$ may have positive, but finite, codimension.

Clearly, $\tilde \H$ is still a Banach space. We next list some of the properties that make it analogous to the Hardy space $\H$. A part in the statement of subsequent Proposition \ref{twist} is played by the weight function
$$
w=P_c(I+(-\Delta)^{-1}V)^{-1} 1 \in L^\infty.
$$

In the free case $w=1$. If $H \geq 0$, by Proposition \ref{ess} $w$ is bounded from below away from $0$, while in general, by Lemma \ref{posit}, $w$ can be replaced by $w_a=w+af_N$, where $f_N$ is the ground state.

The weight $w$ has properties similar to those of the constant function $1$ in the free case. For example, $Hw=0$, implying that
$$
S_tw=\frac {\sin(t \sqrt H)P_c}{\sqrt H} w = t w.
$$
Consequently, $\tilde P_{\leq n} w=w$ and $\tilde P_n w = 0$ for all twisted Littlewood--Paley projections $\tilde P_n$, which are defined in Definition \ref{pw}.

\begin{proposition}\lb{twist}Assume $V \in \mc K_0$, and $H = -\Delta + V$ has no eigenvalue or resonance at zero energy, and no positive eigenvalues. Then $\tilde \H \subset P_c L^1$  and functions in $\tilde \H$ are orthogonal to $w$. The operators $H^{i\sigma}$ and all Mihlin multipliers $m(H)$ for $m \in C^\infty_c((0, \infty))$, including the Paley--Wiener projections $\tilde P_n$, are bounded from $\tilde \H$ to itself.
\end{proposition}

Since $\tilde \H \subset L^1$, by duality $L^\infty \subset \tilde {BMO}$. Elements of $\tilde {BMO}$ are in fact equivalence classes, up to adding arbitrary multiples of $w$ and bound states.

Much more general Mihlin multipliers $m(H)$ are still bounded from $\tilde \H$ to $L^1$:
\begin{theorem}\lb{hor} Assume $V \in \mc K_0 \cap L^{3/2, \infty}$, and $H = -\Delta + V$ has no eigenvalue or resonance at zero energy, and no positive eigenvalues. Suppose $m:(0, \infty) \to \C$ satisfies Mihlin's condition (\ref{cm}) and take $s=2$ in Definition \ref{dhardy}. Then $m(\sqrt H)$ is bounded from the twisted Hardy space $\tilde {\H}$ to $L^1$.
\end{theorem}

Interpolation involving $\tilde \H$ is harder, requiring the development of the full theory of skewed Hardy spaces and skewed $BMO$. We leave this for a future paper. Still, we present one particular useful case.

\begin{lemma}\lb{inter} Assume $V \in \mc K_0 \cap L^{3/2, \infty}$, and $H = -\Delta + V$ has no eigenvalue or resonance at zero energy, and no positive eigenvalues. For $\frac 32 \leq p \leq 2$, $\theta=\frac {2p}{p-1}$, $[\tilde \H, L^2]_\theta=L^p$.
\end{lemma}
By reiteration, all the ``twisted'' interpolation spaces in this range are just the usual Lebesgue and Lorentz spaces.


The following theorem is the partial confirmation of a conjecture made in \cite{becgol3}:
\begin{theorem}\lb{opt} Assume $V \in \K_0$ and let $H=-\Delta+V$ have no eigenvalue or resonance at zero energy, and no positive eigenvalues. Then
\begin{enumerate}
\item $H^sP_c (-\Delta)^{-s} \in \B(L^p)$ and $(-\Delta)^s H^{-s}P_c \in \B(L^p)$ for $\frac 1 {1-s} < p < \frac 1 s$, $0<s<\frac 1 2$.
\item $(-\Delta)^{-s} H^s P_c \in \B(L^p)$ and $P_c H^{-s} (-\Delta)^s \in \B(L^p)$ for $\frac 1 {1-s} < p < \frac{1}{s}$, $0<s<\frac 1 2$.
\end{enumerate}
\end{theorem}
Various results also hold in borderline cases. For $(-\Delta)^s H^{-s}P_c$ and its adjoint we obtained only half the conjectured range of $L^p$ bounds. This is due to using $\tilde \H$ in the proof, whose full interpolation properties are not yet known, instead on $\H$. 

The case $s=\frac12$, $p=2$ is proved in~\cite{goldberg}. See Lemma \ref{hilbert} for easy reference. Assuming only that $V \in \K_0$, the optimal range is conjectured in \cite{becgol3} to be $1<p<\frac 1 s$, with some bounds also holding in the endpoint cases. Theorem \ref{opt} is valid for a restricted $L^p$ range.

For comparison, also see the following result in \cite{hong}:
\begin{lemma}[{\cite[Lemma 1.4]{hong}}] \lb{intertwine} Suppose $V \in\K_0 \cap L^{3/2,\infty}(\R^3)$ and $H = -\Delta+ V$ has no eigenvalue or resonance at zero energy, and no positive eigenvalues. For every $0 \leq s \leq 1$, 
\begin{enumerate}
\item Operators  $P_c H^s (-\Delta)^{-s},\;(-\Delta)^s H^{-s} P_c$ belong to $\B(L^p)$ over the range $1<p<\frac 3 {2s}$.
\item Operators $(-\Delta)^{-s} H^s P_c,\;P_c H^{-s} (-\Delta)^s$ belong to $\B(L^p)$ in the range $\frac 3 {3-2s}<p<\infty$.
\end{enumerate}

Moreover, $P_c H (-\Delta)^{-1},\;\Delta H^{-1} P_c \in \B(L^1)$ and $(-\Delta)^{-1} H P_c,\;P_c H^{-1} \Delta \in \B(L^\infty)$.
\end{lemma}

\subsection{Paley--Wiener decompositions and square functions}
While in \cite{becgol3} we based our argument on complex interpolation, in this paper we also employ real methods, namely Paley--Wiener projections and the square function.

Fix a smooth cutoff function $\phi \in C^\infty_c(\R)$ such that
$$
\sum_{k \in \Z} \phi(2^{-k} t) = \chi_{(0, \infty)}(t)
$$
and $\supp \phi \subset [\frac 1 2, 4]$.

\begin{definition}\lb{pw} The twisted Paley--Wiener ``projections'' $\tilde P_n$, corresponding to the perturbed Hamiltonian $H$, are
$$
\tilde P_n=\phi(2^{-n} \sqrt H),\ \tilde P_{\leq n} = \sum_{n' \leq n} \tilde P_{n'}, \tilde P_{\geq n} = \sum_{n' \geq n} \tilde P_{n'}.
$$
\end{definition}

The operators $\tilde P_{\leq n}$ and $\tilde P_{\geq n}$ can also be defined directly using appropriate cutoff functions.

\begin{definition}
The Littlewood-Paley square function for the perturbed Hamiltonian $H$ is
$$
[S_H f](x) = \bigg(\sum_{n \in \Z} |\tilde P_n f(x)|^2\bigg)^{1/2}.
$$
\end{definition}

The Paley--Wiener projections are not singular integral operators. In the free case, they are given by convolutions with smooth, rapidly decaying functions. Understanding the twisted square function $S_H$ and Paley--Wiener projections $\tilde P_n$ leads to another proof of fractional integration bounds, among other consequences.

Square functions are comparable to the original function in the $L^p$ norms, $1<p<\infty$, as shown in \cite{becgol3}, while different results hold at the endpoints. Here we get rid of the $V \in L^{3/2, \infty}$ condition and also include the endpoint cases.

\begin{proposition}\lb{sqr} Assume $V \in \mc K_0$ and $H=-\Delta+V$ has no eigenvalue or resonance at zero energy, and no positive eigenvalues. Then for each $p \in (1, \infty)$
$$
\|S_H f\|_{L^p} \les_p \|f\|_{L^p} \les_p \|S_H f\|_{L^p}.
$$
Moreover, if $V \in \K_0 \cap L^{3/2, \infty}$, then $\|S_H f\|_{L^1} \les \|f\|_{\tilde \H}$ and $\|S_H f\|_{\tilde {BMO}} \les \|f\|_{L^\infty}$.
\end{proposition}
See \cite{becgol3} for a proof, involving random multipliers, Kintchine's inequality, and Theorem \ref{mainthm}. Half the cases are treated by duality.

The twisted dyadic Paley--Wiener projections $\tilde P_{\leq n}$, $\tilde P_n$, and $\tilde P_{\geq n}$ have non-singular integral kernels and are bounded on each Lebesgue space $L^p$, $1 \leq p \leq \infty$, with norm uniformly bounded in $p$ and $n$, as shown by the following lemma:

\begin{lemma}\lb{aux} Assume $V \in \mc K_0$ and $H=-\Delta+V$ has no eigenvalue or resonance at zero energy, and no positive eigenvalues. Consider a multiplier $m:[0, \infty)\to \C$ such that $m \in C^\infty_c([0, \infty))$, such that $m^{(n)}(0)=0$ for all odd $n$. Then $m(\sqrt H)=m(\sqrt H) P_c$ is uniformly bounded on $L^p$ for $1 \leq p \leq \infty$, with the operator norm controlled by $\|t \partial_t \widehat {\tilde m}(t)\|_{L^1}$, where $\tilde m$ is the even extension of $m$.

Moreover, the integral kernel of $m(\sqrt H)$ has arbitrary decay: $|m(\sqrt H)(x, y)| \les_n |x-y|^{-n}$.
\end{lemma}

More generally, one has that
\begin{proposition}\lb{paleywiener} Assume $V \in \mc K_0$ and $H=-\Delta+V$ has no eigenvalue or resonance at zero energy, and no positive eigenvalues. Then the twisted Paley--Wiener projections $\tilde P_{\leq n}$ and $\tilde P_n$ are bounded from $L^p$ to $L^q$ for any $1\leq p \leq q \leq \infty$ and
$$
\|\tilde P_{\leq n}\|_{\B(L^p, L^q)} \les 2^{2n(\frac 1 p -\frac 1 q)},\ \|e^{i\tau \sqrt H} \tilde P_{\leq n}\|_{\B(L^p, L^q)} \les 2^{2n(\frac 1 p -\frac 1 q)},
$$
and same for $\tilde P_n$.
\end{proposition}

\subsection{Strichartz estimates for the wave equation}
The wave equation with potential, corresponding to the Hamiltonian $H=-\Delta+V$, is
\be\lb{wave}
\partial^2_t u + Hu = F(t,x).
\ee

In this paper we use real methods, such as Paley--Wiener projections and the square function, to prove Strichartz estimates for (\ref{wave}), following \cite{sog} instead of \cite{beals}.

Strichartz estimates have the form
$$
\|u\|_{L^\infty_t \dot H^s_x \cap L^p_t L^q_x} + \|u_t\|_{L^\infty_t \dot H^{s-1}_x} \les \|u_0\|_{\dot H^s} + \|u_1\|_{\dot H^{s-1}} + \|F\|_{L^{\tilde p'}_t L^{\tilde q'}_x},
$$
and hold for appropriate values of $s$, $p$, $q$, $\tilde p$, and $\tilde q$. See (\ref{strichartz_est}) and the ensuing discussion.

Due to the low regularity of the potential $V \in \K_0$, the homogeneous Sobolev spaces $\dot H^s$ and their twisted analogues
$$
\dot {\tilde H}^s = \{f \in \mc S' \mid H^{s/2} f \in L^2\}
$$
only coincide for $|s| \leq 1$, see Lemma \ref{hilbert}, so we confine our analysis to this range.

When $s>1$, Strichartz estimates still hold, but the spaces and operators involve a mixture of the free Laplacian and $H$.

\begin{theorem} Assume that $V \in \K_0$ and $H=-\Delta+V$ has no zero energy eigenstate or resonance, and no positive eigenvalues. Then Strichartz estimates (\ref{strichartz_est}) hold for $P_c f$, the continuous spectrum projection of the solution, in the range $0 \leq s \leq 1$, for all appropriate $p$, $q$, $\tilde p$, and $\tilde q$ satisfying conditions (\ref{condition1}) and (\ref{condition2}).
\end{theorem}

The proof is standard (and we omit it) once we have the following $L^p \to L^{p'}$ decay estimates:
\begin{lemma}\lb{lplemma} Assume $V \in \mc K_0$ and $H=-\Delta+V$ has no eigenvalue or resonance at zero. For $1 \leq p \leq 2$ and $s=\frac 2 p - 1$,
$$
\|\cos(t \sqrt H) H^{-s} P_c f\|_{L^{p'}} \les |t|^{-s} \|f\|_{L^p}.
$$
\end{lemma}


Other than Strichartz estimates, another class of estimates for the wave equation is named after Peral \cite{peral}, who proved for a general wave equation ($H$ is an elliptic operator with second-order terms only, with smooth coefficients constant outside a compact spacetime set) that
$$
\bigg\|\frac {\sin(t\sqrt H)}{\sqrt H} f\bigg\|_{\dot W^{s, p}} \leq C_{p, t} \|f\|_{L^p},
$$
for $\big|\frac 1 p - \frac 1 2\big| \leq \frac 1 {d-1}$, $|s| \leq 1-(d-1)\big|\frac 1 p - \frac 1 2\big|$.

In our context, this result also holds and does not require using the Hardy space or $BMO$ at the endpoints:
\begin{theorem} Assume $V \in \mc K_0$ and $H=-\Delta+V$ has no eigenvalue or resonance at zero. For $1 \leq p \leq \infty$ and $|s| \leq 1-2\big|\frac 1 p - \frac 1 2\big|$,
$$
\bigg\|\frac {\sin(t\sqrt H)}{\sqrt H} f\bigg\|_{\dot W^{s, p}} \les |t|^{1-s} \|f\|_{L^p}.
$$
\end{theorem}
This is a straightforward consequence of Lemma \ref{newversion}, for $p=1$ and $p=\infty$, and can be obtained by interpolating with the $p=2$ energy bound for all other $p$.

\subsection{Notations}
We use the following standard notations for the Fourier transform and its inverse:
$$
[\mc F f](\xi) = \widehat f(\xi) = \int_{\R^d} f(x) e^{-ix\xi} \dd x
$$
and
$$
[\mc F^{-1} g](x) = \frac 1 {(2\pi)^d} \int_{\R^d} g(\xi) e^{ix\xi} \dd \xi.
$$
Also, let
\begin{itm}
\item $a \les b$ mean that there exists a constant $C$ independent of $a$ and $b$ such that $a \leq Cb$
\item $\langle \cdot \rangle = (1+(\cdot)^2)^{1/2}$
\item $\one$ denote an indicator function
\item $f_+=\max(f, 0)$, $f_-=\min(f, 0)$
\item $\B(X, Y)$ be the space of bounded linear operators from $X$ to $Y$, $\B(X)$ be the space of bounded operators from $X$ to itself
\item $\M$ be the space of complex-valued measures of finite variation
\item $\mu$ be the usual Lebesgue measure
\item $C$ be the space of continuous (and sometimes bounded) functions
\item $C_0=\{f \in C \mid \lim_{x \to \infty} f(x)=0\}$
\item $\H=\H^1$ be the Hardy space; $\dot W^{s, p}$, $\dot H^s$ be homogeneous Sobolev spaces
\item $R_0(\lambda):=(-\Delta-\lambda)^{-1}$, $R_V(\lambda):=(H-\lambda)^{-1}$, $R^\pm(\lambda)=R(\lambda \pm i0)$
\item
$S_{0t}(x, y):=\chi_{t \geq 0} \frac {\sin(t\sqrt{-\Delta})}{\sqrt{-\Delta}}(x, y)= \frac {\chi_{t \geq 0}} {4\pi t} \delta_t(|x-y|)$, $S_t:=\chi_{t \geq 0} \frac {\sin(t \sqrt H) P_c}{\sqrt H}$
\item $C_{0t}:=\chi_{t \geq 0} \frac {\cos(t \sqrt {-\Delta}) P_c}{-\Delta}$, $C_t:=\chi_{t \geq 0} \frac {\cos(t \sqrt H) P_c}H$.
\end{itm}

\section{Preliminaries}
Here we collect several auxiliary results that, while interesting in their own right, are less important than the main results.

Recall from \cite{becgol2} the connection between the sine propagator for the wave equation and the Fourier transform of the resolvent, for a given time-independent Hamiltonian: with the notations defined above,
$$
\mc F S_{0t} = R_0^+(\lambda^2) \text{ and } \mc F S_t = R_V^+(\lambda^2) P_c,
$$
hence $\mc F R_V^+(\lambda^2) P_c=2\pi S_{-t}$ and likewise $\mc F R_V^-(\lambda^2) P_c=2\pi S_t$.

Concerning the point spectrum component, as shown in \cite{becgol3},
$$
R_V(\lambda^2) P_p = -\sum_{j=1}^J \frac 1 {\lambda^2+\mu_j^2} P_j,\ \mc F[R_V^+((\cdot)^2) P_p](t) = -\sum_{j=1}^J \frac \pi{\mu_j} e^{-\mu_j |t|} P_j,
$$
so in particular, for $\lambda \in \R$, $R_V^+(\lambda^2) P_p$ and $R_V^-(\lambda^2) P_p$ cancel each other ($R_V P_p$ is analytic, hence continuous, in a neighborhood of the continuous spectrum) and do not contribute to the continuous part of the spectral measure.

Let the negative eigenvalues of $H$ be $0>\lambda_1>\ldots>\lambda_N$ with corresponding eigenfunctions $f_n$ for $1 \leq n \leq N$. It is shown in \cite{becgol2} that
$$
|f_n(x)| \les \frac {e^{-\mu_j|x|}}{\langle x \rangle},
$$
where $\lambda_n=-\mu_n^2$, $\mu_n>0$. We assume the eigenfunctions are $\dot H^1$- or $L^2$-normalized.

Sometimes it is more convenient to consider the continuous and point spectrum pieces separately (or only the first), while at other times we consider them together. Either way, sine propagator estimates correspond to resolvent estimates by means of the Fourier transform.

Recall a fundamental result from \cite[Theorem 1]{becgol2}
\be\lb{fund}
\sup_{x, y \in \R^3} \int_0^\infty t \Bigl|\frac{\sin(t\sqrt{H})P_c}{\sqrt{H}} (x,y)\Bigr| \dd t < \infty
\ee
and the following important estimate from \cite{becgol3}:
\begin{equation}\lb{funda}
\int_0^\infty 
\Bigl|\frac{\sin(t\sqrt{H})P_c}{\sqrt{H}} (x,y)\Bigr|\,dt \les \frac 1 {|x-y|}.
\end{equation}
Due to the finite speed of propagation of solutions to the wave equation,
\be\lb{outside}
\one_{> t}(|x-y|) S_t(x, y) = -\one_{>t}(|x-y|) \sum_{j=1}^J \frac {\sinh(t\mu_j)}{\mu_j} f_j(x) \ov{f_j}(y).
\ee

With our notations, $C_{0t} = \int_t^\infty S_{0\tau} \dd \tau$ and likewise
$$
C_t = \int_t^\infty S_\tau \dd \tau.
$$
Among other estimates (\ref{fund}) implies
\be\lb{est100}
\sup_{x, y \in \R^3} \int_0^\infty |C_t(x, y)| \dd t < \infty.
\ee

We continue by proving some estimates for $C_t$ that were omitted from the long list in Theorem 1 in \cite{becgol2}, starting with the free case.

\begin{lemma}\lb{free_case}
The following identities and sharp inequalities are true in the free case for $t \geq 0$:
$$
C_{0t}(x, y)=\frac 1 {4\pi} \frac {\chi_{\geq t}(|x-y|)}{|x-y|},\ |C_{0t}(x, y)| \leq \frac 1 {4\pi |x-y|}, 
$$
and
$$
\sup_{x, y} \bigg|\frac {\cos(t \sqrt{-\Delta})}{-\Delta}(x, y)\bigg| = \frac 1 {4 \pi t}.
$$
Moreover
$$
\int_0^\infty |\nabla C_{0t}(x, y)| \dd t = \frac 1 {2\pi |x-y|},\ \int_0^\infty t |\nabla C_{0t}(x, y)| \dd t = \frac 3 {8\pi}.
$$
\end{lemma}
\begin{proof} This follows by an explicit computation, noting that
\be\lb{nablacos}
|\nabla C_{0t}(x, y)| = \frac 1 {4\pi} \bigg(\frac {\delta_t(|x-y|)}{|x-y|} + \frac {\chi_{\geq t}(|x-y|)}{|x-y|^2} \bigg).
\ee
\end{proof}

These identities and sharp inequalities in the free case lead to similar inequalities in the perturbed case:

\begin{lemma}\lb{lemma21} Assume $V \in \K_0$, and $H = -\Delta + V$ has no eigenvalue or resonance at zero, and no positive eigenvalues. Then
\be\lb{1stcosbd}
\sup_{x, y} \bigg|\frac {\cos(t \sqrt H)P_c} H(x, y)\bigg| \les |t|^{-1}
\ee
and
\be\lb{cosbounds}\begin{aligned}
\int_0^\infty |\nabla_y \cos(t\sqrt H) P_c (-\Delta)^{-1}(x, y)| \dd t &\les \frac 1 {|x-y|},\\
\sup_{x, y} \int_0^\infty t|\nabla_y \cos(t\sqrt H) P_c (-\Delta)^{-1}(x, y)| \dd t &< \infty.
\end{aligned}\ee
\end{lemma}
The modified cosine kernel 
$$
\cos(t \sqrt H) P_c (-\Delta)^{-1} = C_t H (-\Delta)^{-1}
$$
has better properties than $C_t$. The estimates in (\ref{cosbounds}) only hold for this modified cosine kernel, which complicates some arguments later on.
\begin{proof} From (\ref{funda}) it follows that
$$
\frac {\cos(t \sqrt H)P_c} H(x, y) = \int_t^\infty \frac {\sin(\tau \sqrt H) P_c}{\sqrt H}(x, y) \dd \tau \les \sup_{\tau \geq t} \tau^{-1} \int_0^\infty \bigg| \frac {\tau \sin(\tau \sqrt H) P_c}{\sqrt H}(x, y) \bigg| \dd \tau \leq t^{-1}.
$$
Another proof uses a similar inequality from \cite{becgol2}:
$$
\|\cos(t \sqrt H)P_c f\|_{L^\infty} \les t^{-1} \|\Delta f\|_{L^1}
$$
or equivalently
$$
\|\cos(t \sqrt H)P_c (-\Delta)^{-1} f\|_{L^\infty} \les t^{-1} \|f\|_{L^1}.
$$
Then $\frac {\cos(t \sqrt H)P_c} H = [\cos(t \sqrt H)P_c (-\Delta)^{-1}] [(-\Delta) H^{-1}]$ and $-\Delta H^{-1} = (I+V(-\Delta)^{-1})^{-1} \in \B(L^1)$ when $V \in \K_0$.

However, the gradient bounds (\ref{cosbounds}) are only preserved by applying an $L^1$-bounded operator to the left, not to the right, so $(-\Delta)^{-1}$ and $H^{-1}$ are not interchangeable there. These bounds only hold as stated. To prove them, start from
\be\lb{ident}
\cos(t\sqrt H) P_c (-\Delta)^{-1} = P_c C_{0t} - \int_0^t S_{\tau_1} V C_{0t-\tau_1} \dd \tau_1.
\ee




The first term on the right-hand side already satisfies the desired estimates and for the second term $(**)$ we obtain that
$$\begin{aligned}
(**)(x, y)=&\nabla_y \int_0^t S_{\tau_1} V \frac {\cos((t-\tau_1)\sqrt{-\Delta})} {-\Delta} \dd \tau_1 \\
\les& \int_0^t \int_{\R^3} |S_{\tau_1}(x, z)| |V(z)| \bigg|\nabla_y \frac {\cos((t-\tau_1)\sqrt{-\Delta})} {-\Delta}(z, y)\bigg| \dd z \dd \tau_1.
\end{aligned}$$
The bounds in (\ref{cosbounds}) follow from the cosine estimates for the free propagator in Lemma \ref{free_case} and from (\ref{fund}) and (\ref{funda}):
$$\begin{aligned}
\int_0^\infty t (**)(x, y) \dd t \les& \int_{\R^3} \int_0^\infty \tau_1 |S_{\tau_1}(x, z)| \dd \tau_1 |V(z)| \int_0^\infty \bigg|\nabla_y \frac {\cos(\tau_2\sqrt{-\Delta})} {-\Delta}(z, y)\bigg| \dd \tau_2 \dd z + \\
&\int_{\R^3} \int_0^\infty |S_{\tau_1}(x, z)| \dd \tau_1 |V(z)| \int_0^\infty \tau_2 \bigg|\nabla_y \frac {\cos(\tau_2\sqrt{-\Delta})} {-\Delta}(z, y)\bigg| \dd \tau_2 \dd z\\
\les & \int_{\R^3} |V(z)| \bigg(\frac 1 {|x-z|} + \frac 1 {|z-y|}\bigg) \dd z \leq 2 \|V\|_\K
\end{aligned}$$
and
$$\begin{aligned}
\int_0^\infty (**)(x, y) \dd t \les& \int_{\R^3} \int_0^\infty |S_{\tau_1}(x, z)| \dd \tau_1 |V(z)| \int_0^\infty \bigg|\nabla_y \frac {\cos(\tau_2\sqrt{-\Delta})} {-\Delta}(z, y)\bigg| \dd \tau_2 \dd z \\
\les & \int_{\R^3} \frac {|V(z)| \dd z} {|x-z| |z-y|} \leq \frac {\|V\|_\K}{|x-y|}.
\end{aligned}$$
\end{proof}

We are still interested in such gradient bounds related to $C_t$, even though $C_t$ itself does not fulfill them. Now the starting point is the identity
$$
C_t = P_c C_{0t} - \int_0^t S_{t-s} V C_{0s} ds - C_t V (-\Delta)^{-1}.
$$
The difference from above is the special term $C_tV(-\Delta)^{-1}$, which does not fulfill (\ref{cosbounds}). Still, we can split this term
\be\lb{c1t}\begin{aligned}
C_tV(-\Delta)^{-1} &= \int_{\R^3} C_t(x, z) \frac {V(z) \dd z}{4\pi |z-y|} \\
&= \int_{|z-y| \geq \frac {|x-y|} 2} C_t(x, z) \frac {V(z) \dd z}{4\pi |z-y|} + \int_{|z-y| < \frac {|x-y|} 2} C_t(x, z) \frac {V(z) \dd z}{4\pi |z-y|} =: \tilde C_{1t} + \tilde C_{2t}
\end{aligned}\ee
into two parts, $\tilde C_{1t}$ and $\tilde C_{2t}$, according to whether $|z-y| \geq \frac {|x-y|} 2$ or not (a possible alternative split is according to whether $|x-z|<|z-y|$). Then $\tilde C_{1t}$ satisfies one of the bounds in (\ref{cosbounds}):
$$
\int_0^\infty |\nabla_y \tilde C_{1t}(x, y)| \dd t \les \int_{|z-y| \geq \frac {|x-y|} 2} \int_0^\infty |C_t(x, z)| \dd t \frac {|V(z)| \dd z}{|z-y|^2} \les \frac 1 {|x-y|} \int_{\R^3} \frac {|V(z)| \dd z}{|z-y|} \les \frac {\|V\|_\K}{|x-y|}.
$$
This follows using (\ref{est100}).

We continue by proving a generalization of Lemma 4.2 from \cite{becgol3}.
\begin{lemma} \label{newversion}
Assume $V \in \K_0$ and $H = -\Delta + V$ has no eigenvalue or resonance at
zero, and no positive eigenvalues.  The sine propagator
satisfies the bounds
\be\lb{eq:conicalbound}
\int_{\R^3} \Big| \frac {\sin(t \sqrt H) P_c}{\sqrt H}(x, y)\Big|\,dx \les t,\ \int_{\R^3} \frac 1 {|x-y_0|} \Big| \frac {\sin(t \sqrt H) P_c}{\sqrt H}(x, y)\Big|\,dx \les 1,
\ee
uniformly in $y, y_0 \in \R^3$, $t > 0$.
\end{lemma}

A version of the former inequality is proved in Lemma 4.2 in \cite{becgol3}. The same inequalities apply to the modified sine kernel $\sin(t \sqrt H) \sqrt H P_c (-\Delta)^{-1}$, because they are preserved by composition to the right with an $L^1$-bounded operator.

\begin{proof} Use resolvent identities to write
\begin{equation*}
R_V^+(\lambda^2)P_c = R_0^+(\lambda^2)P_c - R_0^+(\lambda^2) VR_V^+(\lambda^2)P_c.
\end{equation*}
Recalling that $R_0^+(\lambda^2)(x,y) = \frac{e^{i\lambda|x-y|}}{4\pi|x-y|}$,
there is an explicit formula in the free case
\begin{equation*}
\mathcal F[R_0^+((\cdot)^2)](t, x, y) = \chi_{t \geq 0} \frac {\sin(t\sqrt{-\Delta})}{\sqrt{-\Delta}}(x, y) =  S_{0t}(x, y) = \frac{\delta_t(|x-y|)}{4\pi|x-y|},
\end{equation*}
which satisfies the first inequality in~\eqref{eq:conicalbound}, as an identity with a constant of $1$. In the second inequality in (\ref{eq:conicalbound}), in the free case, the maximum of $1$ is attained when $|y-y_0| \leq t$.

To extend these bounds to $\mc F [R_0^+(\cdot^2)] P_c$, again note that both inequalities in (\ref{eq:conicalbound}) are stable under composition to the right with an $L^1$-bounded operator and that $P_c \in \B(L^1)$.

It is a consequence of~\eqref{funda}, or \cite[Theorem 1]{becgol2}, that
for each fixed $y \in \R^3$, the integral kernel
$$
\mathcal F(VR_V^+((\cdot)^2)P_c)(t,x,y)=\chi_{t \geq 0}V(x)\frac {\sin(t \sqrt H) P_c}{\sqrt H}(x, y)=V(x) S_t(x, y)
$$
is a measure
in $\R^{1+3}$ with finite total variation.  It follows that 
\begin{multline*} 
\int_{\R^3} \bigg| \frac {\sin(t\sqrt H)P_c}{\sqrt H}(x, y) - \frac {\sin(t\sqrt{-\Delta})}{\sqrt{-\Delta}}P_c(x, y)\bigg|\,dx
\\
\begin{aligned}
&
\leq \int_{\R^3} \int_0^t  \int_{\R^3}
\bigg|\frac {\sin((t-s)\sqrt{-\Delta})}{\sqrt{-\Delta}}(x,w)\bigg|\,
\bigg|V(w) \frac {\sin(s\sqrt H)P_c}{\sqrt H}(w, y)\bigg|\,dw ds dx 
\\
&\les \sup_{0 \leq s \leq t} \sup_{w\in\R^3} \int_{\R^3} \bigg|\frac {\sin((t-s)\sqrt{-\Delta})}{\sqrt{-\Delta}}(x,w)\bigg|\,dx \les t.
\end{aligned}
\end{multline*}

The other inequality in (\ref{eq:conicalbound}) follows in a completely analogous manner in the perturbed case from the free case described above:
\begin{multline*} 
\int_{\R^3} \frac 1 {|x-y_0|} \bigg| \frac {\sin(t\sqrt H)P_c}{\sqrt H}(x, y) - \frac {\sin(t\sqrt{-\Delta})}{\sqrt{-\Delta}}P_c(x, y)\bigg|\,dx
\\
\begin{aligned}
&\leq \int_{\R^3} \int_0^t  \int_{\R^3} \frac 1 {|x-y_0|}
\bigg|\frac {\sin((t-s)\sqrt{-\Delta})}{\sqrt{-\Delta}}(x,w)\bigg|\,
\bigg|V(w) \frac {\sin(s\sqrt H)P_c}{\sqrt H}(w, y)\bigg|\,dw ds dx \\
&\les \sup_{0 \leq s \leq t} \sup_{w\in\R^3} \int_{\R^3} \frac 1 {|x-y_0|} \bigg|\frac {\sin((t-s)\sqrt{-\Delta})}{\sqrt{-\Delta}}(x,w)\bigg|\,dx \les 1.
\end{aligned}
\end{multline*}
\end{proof}

We shall need similar estimates for the modified cosine kernel $\nabla_y \cos(t \sqrt H)P_c (-\Delta)^{-1}$. We also consider such estimates for the ``good'' component $\tilde C_{1t}$ (\ref{c1t})
$$
\tilde C_{1t} = \int_{|z-y| \geq \frac {|x-y|} 2} C_t(x, z) \frac {V(z) \dd z}{|z-y|}.
$$
\begin{lemma} Assume $V \in \K_0$, and $H = -\Delta + V$ has no eigenvalue or resonance at
zero, and no positive eigenvalues.  Then, for $c>0$, the modified cosine kernel $\nabla_y \cos(t \sqrt H)P_c (-\Delta)^{-1}$ and $\tilde C_{1t}$ satisfy the bound
\be\lb{eq:conicalbound2}
\int_{|x-y| \leq ct} |\nabla_y \cos(t \sqrt H)P_c (-\Delta)^{-1}(x, y)|\,dx \les_c t,\
\int_{|x-y| \leq ct} |\nabla_y \tilde C_{1t}(x, y)|\,dx \les_c t,
\ee
uniformly in $y \in \R^3$, $t > 0$.
\end{lemma}
The unmodified kernel $\frac {\cos(t \sqrt H) P_c} H$ need not satisfy these bounds, because, unlike (\ref{eq:conicalbound}), they are not preserved by composition to the right with $L^1$-bounded operators. However, one piece of the difference between the modified and unmodified kernels, namely $\tilde C_{1t}$ (\ref{c1t}), does.
\begin{proof}
Taking the gradient in the identity (\ref{ident})
$$
\nabla_y\cos(t\sqrt H) P_c (-\Delta)^{-1}(x, y) = \nabla_y C_{0t}(x, y) - \int_0^t S_{\tau_1} V \nabla_y C_{0t-\tau_1}(z, y) \dd \tau_1.
$$
The first right-hand side term, corresponding to the free Laplacian, is given by the formula (\ref{nablacos})
$$
|\nabla_y C_{0t}(x, y)| = \frac 1 {4\pi} \bigg(\frac {\delta_t(|x-y|)}{|x-y|} + \frac {\chi_{\geq t}(|x-y|)}{|x-y|^2} \bigg)
$$
This satisfies the required bounds: for $c<1$ the integral vanishes and for $c \geq1$ 
$$
\int_{|x-y| \leq t} |\nabla_y C_{0t}(x, y)| \dd x = t,\ \int_{|x-y| \leq ct} |\nabla_y C_{0t}(x, y)| \dd x = ct.
$$
For the second term $(**)$, by Lemma \ref{newversion} we obtain
$$\begin{aligned}
\int_{|x-y| \leq ct} (**)(x, y) \dd x =& \int_{|x-y| \leq ct} \int_0^t \int_{|z-y| \leq t} S_{\tau_1}(x, z) V(z) \nabla_y C_{0t-\tau_1}(z, y) \dd z \dd \tau_1 \dd x + \\
&\int_{|x-y| \leq ct} \int_0^t \int_{|z-y| > t} S_{\tau_1}(x, z) V(z) \nabla_y C_{0t-\tau_1}(z, y) \dd z \dd \tau_1 \dd x \\
\leq& \int_{\R^3} \sup_{\tau_1 \leq t} \int_{|x-z| \leq (c+1)t} S_{\tau_1}(x, z) \dd x \cdot V(z) \int_\R |\nabla_y C_{0\tau_2}(z, y)| \dd \tau_2 \dd z + \\
&\int_{|x-y| \leq ct} \int_{|z-y|>t} \int_0^t S_{\tau_1}(x, z) \dd \tau_1 \cdot V(z) \sup_{\tau_2 \leq t} |\nabla_y C_{0\tau_2}(z, y)| \dd z \\
\les& t \int_{\R^3} \frac {V(z) \dd z}{|y-z|} + \int_{|x-y| \leq ct} \int_{|z-y|>t} \frac {V(z) \dd z}{|x-z| |z-y|^2} \dd x\\
\les& t\|V\|_\K + t^{-1} \|V\|_\K \int_{|x-y| \leq ct} \frac {dx}{|x-y|} \les t.
\end{aligned}$$

Regarding $\tilde C_{1t}$ (\ref{c1t}), note that
$$
|\nabla_y \tilde C_{1t}(x, y)| \les \int_{|z-y| \geq \frac {|x-y|} 2} |C_t(x, z)| \frac {|V(z)| \dd z}{|z-y|^2} \les \frac 1 {|x-y|} \int_{\R^3} \frac {|V(z)| \dd z}{|x-z| |z-y|} \les \frac {\|V\|_\K}{|x-y|^2},
$$
which immediately implies the desired bound.
\end{proof}

\section{Maximum principles}

Let $f \in C^2$ be a classical solution of the boundary-value problem
\be\lb{bvproblem}
-\Delta f + V f = g \geq 0,\ \liminf_{x \to \infty} f(x) \geq 0.
\ee

If $V \geq 0$, then by the classical (weak) \textbf{maximum principle} $f \geq 0$. Indeed, suppose at first that $V>0$. If $f \in C^2$ and $\inf f<0$, then due to the boundary condition the minimum is achieved at some $x_0 \in \R^3$, where $f(x_0)<0$. But then $[-\Delta f](x_0) \leq 0$ and $[Vf](x_0) < 0$, producing a contradiction. The condition can then be relaxed to $V \geq 0$ by approximating $V$ in $\K_0$.

The weak maximum principle applies more generally. For example, take $V$ not necessarily positive, but such that $\|V_-\|_\K < 4\pi$, where $V_-=\min(V, 0)$, and assume that $g \geq 0 \in \dot H^{-1}$ and $f \in \dot H^1$, implying that $f$ vanishes at infinity. Multiplying the equation by $f_-$ we get
$$
\|\nabla f_-\|_{L^2}^2 + \langle V_+ f_-, f_- \rangle = -\langle V_- f_-, f_- \rangle + \langle g, f_- \rangle.
$$
Taking into account the fact that $|\langle V_- f, f \rangle| \leq (4\pi)^{-1} \|\nabla f\|_{L^2}^2$, obtained by interpolation between $\|f_1 f_2\|_{\K^*} \leq (4\pi)^{-1} \|f_1\|_{(-\Delta)^{-1} L^1} \|f_2\|_{L^\infty}$ and the symmetric inequality, and that $\langle g, f_- \rangle \leq 0$, we see that $f_-=0$, so $f \geq 0$ almost everywhere.

\begin{definition} We say that $V$ is locally in the Kato class, $V \in \K_{loc}$, if for any $y \in \R^3$ there exists $R>0$ such that
\be\lb{conditie}
\int_{B(y, R)} \frac {|V(x)| \dd x}{|x-y|} < \infty.
\ee
\end{definition}

This assumption is strictly weaker than the local Kato condition
$$
\lim_{\epsilon \to 0} \sup_{y \in \R^3} \int_{B(y, \epsilon)} \frac {|V(x)| \dd x}{|x-y|} = 0
$$
or than $\|V\|_\K<\infty$. Condition (\ref{conditie}) equivalent to assuming that for each compact set $C$ and $y \in \R^3$
$$
\int_C \frac {|V(x)| \dd x}{|x-y|} < \infty.
$$
Then the Kato integral (\ref{conditie}) is finite for every $y$ and $R$, but not necessarily uniformly bounded.

Assuming that $\liminf_{x \to \infty} f(x)>0$, the \textbf{strong maximum principle} asserts that $f>0$. We shall prove it under the assumption that $f \geq 0$ and $V \in \K_{loc}$. 

\begin{proposition}[A priori strong maximum principle]\lb{strongmax} For $V \in \K_{loc}$, let $f \geq 0 \in (\Delta^{-1} \K_0)_{loc}$ be a solution of the boundary-value problem (\ref{bvproblem}). If $f \not \equiv 0$, then $f>0$.
\end{proposition}

This result is conditional. We require that $f \geq 0$, but not that $V \geq 0$.

In this setup, by Lemma A.3 in \cite{becgol4}, $\Delta^{-1} \K_0 \subset C_0$ (the set of continuous functions that vanish at infinity), $(\Delta^{-1} \K_0)_{loc} \subset C$, and both $f$ and $\nabla f$ are integrable on surfaces, so the integrals used in the proof will be finite.

\begin{proof}
Clearly, the zero set of $f$ is closed. We shall prove that it is also open. Then either $f$ is constantly $0$ or it never vanishes, which is the desired conclusion.

We already know that $f \geq 0$. Suppose that $f(x_0)=0$ for some $x_0 \in \R^3$. Fix $R$ such that by abuse of notation
$$
\|V\|_\K=\int_{B(x_0, R)} \frac {|V(x)|}{|x-x_0|} < \infty.
$$

Without loss of generality, take $x_0=0$ and let $r=|x|\leq R$ (here $R$ can be arbitrarily large), $f_r=\frac x {|x|} \cdot \nabla f$.

The inequality (\ref{bvproblem}) can be written as $\Delta f \leq Vf$. Taking into account that $f(x_0)=\nabla f(x_0)=0$, by the Green--Stokes theorem for $f$ and $g=r^{-1}$ we obtain that
$$
-\int_{\partial B(0, r)} \frac f {r^2} + \frac {f_r}{r} = -\int_{B(0, r)} \frac {\Delta f} r \geq -\int_{B(0, r)} \frac {Vf} r.
$$
In radial coordinates, this becomes
$$
\partial_r [r \int_{S^2} f(r\omega) \dd \omega] \leq \int_0^r \rho \int_{S^2} V(\rho\omega) f(\rho \omega) \dd \omega \dd \rho.
$$
Doing the same for $\log(f+\epsilon)$ instead of $f$, where $\epsilon>0$ is for regularization, on $A:=B(0, r) \setminus B(0, r_0)$ where $r>r_0>0$, leads to
$$
\int_{\partial B(0, r)} \frac {\log (f+\epsilon)} {r^2} + \frac {[\log (f+\epsilon)]_r}{r} = \int_{\partial B(0, r_0)} \frac {\log (f+\epsilon)} {r^2} + \frac {[\log (f+\epsilon)]_r}{r} + \int_A \frac {\Delta \log(f+\epsilon)}{r}
$$
and, since
$$
\Delta \log(f+\epsilon)=-\frac {|\nabla f|^2}{(f+\epsilon)^2}+\frac{\Delta f}{f+\epsilon} \leq -\frac {|\nabla f|^2}{(f+\epsilon)^2} + V,
$$
to
\be\lb{main_ineq}
\partial_r [r \int_{S^2} \log(f(r\omega)+\epsilon) \dd \omega] + \int_A \frac{|\nabla f|^2}{r(f+\epsilon)^2} \leq \partial_r [r \int_{S^2} \log(f(r\omega)+\epsilon) \dd \omega] \bigg|_{r=r_0} + \int_A \frac V r.
\ee

The last term is uniformly bounded by $\|V\|_\K$. As we take $r_0$ to $0$, the other right-hand side term is of size $\log \epsilon$, since
$$
r \int_{S^2} \log(f(r\omega)+\epsilon) \dd \omega \mid_{r \downarrow 0} = 0,\ \partial_r [r \int_{S^2} \log(f(r\omega)+\epsilon) \dd \omega] \mid_{r \downarrow 0} = 4\pi \log \epsilon.
$$

This term is negative and dominates if $\epsilon$ is sufficiently small. Letting $\epsilon \to 0$, we can make the left-hand side arbitrarily negative, obtaining
\be\lb{ave}
\int_{S^2} \log (f(r\omega)+\epsilon) \dd \omega \leq \|V\|_\K+4\pi \log \epsilon.
\ee
It follows that for all $r>0$
$$
\lim_{\epsilon \downarrow 0} \int_{S^2} \log (f(r\omega)+\epsilon) \dd \omega = \int_{S^2} \log f(r\omega) \dd \omega = -\infty.
$$
Hence, for any threshold $\delta>0$, $f<\delta$ on a set of positive measure, arbitrarily far from $x_0$.

However, this can be strengthened further. Consider any sufficiently small $\epsilon>0$ for which the right-hand side in (\ref{main_ineq}) is negative. Then
$$
\partial_r [r \int_{S^2} \log(f(r\omega)+\epsilon) \dd \omega] + \int_{B(0, r)} \frac{|\nabla f|^2}{r(f+\epsilon)^2} \leq 0.
$$
Integrating in $r$ from $0$ to $R$ leads to
$$
R \int_{S^2} \log(f(R\omega)+\epsilon) \dd \omega + \int_{B(0, R)} \frac{(R-r)|\nabla f|^2}{r(f+\epsilon)^2} \leq 0.
$$
Thus
$$
\int_{B(0, R)} \frac{(R-r)|\nabla f|^2}{r(f+\epsilon)^2} \les -R \log \epsilon,
$$
implying for a different constant
$$
\int_{B(0, R/2)} \frac {|\nabla f|^2}{r (f+\epsilon)^2} \les -\log \epsilon.
$$
Such results are also true on a small neighborhood of $x_0$ with $\epsilon<1$ arbitrarily close to $1$.

By the endpoint Sobolev embedding $\dot H^1 \subset BMO$ in two dimensions, we obtain that
\be\lb{in2}
\int_0^{R/2} r^{-1} \|\log(f(r\omega)+\epsilon)\|_{BMO_\omega}^2 \dd r \les -\log \epsilon.
\ee
In other dimensions $BMO$ is replaced by $L^{\frac {2d}{d-2}}$, but we still get a bound on deviations from the mean.

Then, on $[0, R]$, $\|\log(f(r\omega)+\epsilon)\|_{BMO_\omega}$ must be arbitrarily small on a set of positive measure. More precisely, for any sufficiently small $\alpha<\alpha_0$ and some $c$ independent of $\epsilon$
$$
\mu(\{r \in [0, 2c\alpha] \mid \|\log(f(r\omega)+\epsilon)\|_{BMO}^2 \leq -\alpha \log \epsilon\}) \geq c \alpha.
$$
The bound (\ref{in2}) further implies that, for any sequence $\epsilon_n \to 0$  and almost every $r \in [0, R]$,
\be\lb{in3}
\|\log(f(r\omega)+\epsilon_n)\|_{BMO_\omega} \leq C(r) (-\log \epsilon_n)^{1/2}.
\ee
Indeed, this inequality failing on a set of positive measure $A \subset [0, R]$ implies that for any constant $C_0$, each $r \in A$, and infinitely many $n$ depending on $r$
$$
\|\log(f(r\omega)+\epsilon_n)\|_{BMO_\omega} \geq C_0 (-\log \epsilon_n)^{1/2}.
$$
Passing to the limit, we obtain that $\|\log f(r\omega)\|_{BMO_\omega}=+\infty$ for $r \in A$, which contradicts (\ref{in2}).

Taking $\epsilon_n=e^{-n}$ in the above further implies that (\ref{in3}) holds for every $0<\epsilon\leq 1$, with a different constant.

Recall that by (\ref{ave})
$$
\dashint_{S^2} \log(f(r\omega)+\epsilon) \leq \frac {\|V\|_\K}{4\pi} +\log\epsilon.
$$
Hence, by the John--Nirenberg (or Moser--Trudinger) inequality, for almost every $r \in [0, R]$
$$
\mu(\{\omega \in S^2 \mid \log(f(r\omega)+e^{-n}) \geq (4\pi)^{-1} \|V\|_\K + \log \epsilon + \delta\}) \leq C_1 e^{-\frac{C_2\delta}{C(r)(-\log \epsilon)^{1/2}}}.
$$
Setting $\delta=-(\log \epsilon)/2$, we obtain that
$$
\mu(\{\omega \in S^2 \mid \log(f(r\omega)) \geq (4\pi)^{-1} \|V\|_\K + \frac {\log \epsilon} 2 \}) \leq C_1 e^{-C_2(-\log \epsilon)^{1/2}/C(r)},
$$
which implies that $f(r\omega)=0$ almost everywhere.
\end{proof}


It helps to know that the ground state, meaning the eigenfunction corresponding to the most negative eigenvalue, has constant sign and does not vanish:
\begin{lemma}[The ground state]\lb{ground} Let $\lambda_N=\min(\sigma(H))<0$ be the lowest eigenvalue of $H=-\Delta+V$. Then a strictly positive eigenfunction (the ground state) spans $P_N \dot H^1$, which consequently has dimension $1$.
\end{lemma}
\begin{proof} Let $f_N \in P_N \dot H^1$ be an eigenfunction corresponding to the lowest eigenvalue $\lambda_N<0$, satisfying the equation $H f_N = \lambda_N f_N$. Multiplying the equation by $f_{N+}$ we get
$$
\langle HP_c f_{N+}, f_{N+} \rangle = -\langle H P_p f_{N+}, f_{N+} \rangle + \lambda_N \langle f_{N+}, f_{N+} \rangle.
$$
The left-hand side is non-negative. Moreover, $H P_p$ is negative definite and for any $g$
$$
0 \geq \langle H P_p g, g \rangle \geq \lambda_N \langle g, g \rangle,
$$
with equality only possible when $g$ is an eigenfunction for $\lambda_N$; hence the right-hand side is non-positive.

Thus both $f_{N+}$ and $f_{N-}=f_N-f_{N+}$ are eigenfunctions corresponding to $\lambda_N$. Since $f_N=f_{N+}+f_{N-}$ and both $f_+$ and $f_-$ have constant sign, indeed this eigenspace is spanned by eigenfunctions of constant sign.

If $f_N \geq 0$, $f_N \not \equiv 0$, is such an eigenfunction of constant sign, rewrite its equation as
$$
-\Delta f_N + (V - \lambda_N) f_N = 0.
$$
By the strong maximum principle Proposition \ref{strongmax}, $f_N>0$.

Consider two linearly independent ground states $f_{N1}>0$ and $f_{N2}>0$. If $f_{N1}<\alpha f_{N2}$ for some $\alpha>0$, for the infimum of such $\alpha$ equality must be attained. However, $\alpha f_{N2}-f_{N_1} \geq 0$ is an eigenstate; by the strong maximum principle, either $\alpha f_{N2} - f_{N1}>0$ or $\alpha f_{N2}-f_{N1} \equiv 0$, contradiction.

Discarding the opposite case as well, the remaining possibility is that $f_{N1}-f_{N2}$ takes both positive and negative values. But then $(f_{N1}-f_{N2})_+ \geq 0$ and $(f_{N1}-f_{N2})_- \leq 0$ are both nontrivial eigenfunctions:
$$
H (f_{N1}-f_{N2})_\pm = \lambda_N (f_{N1}-f_{N2})_\pm.
$$
By the strong maximum principle Proposition \ref{strongmax}, $(f_{N1}-f_{N2})_\pm$ never vanish, so each one's support must be the whole of $\R^3$. This is another contradiction, hence the ground state is unique.
\end{proof}

Finally, this is the reason why we need the maximum principle.

\begin{proposition}\lb{ess}
Let $V \in \K_0$ and assume that $H=-\Delta+V \geq 0$ has no zero energy bound states. Let
$$
w=(I+(-\Delta)^{-1}V)^{-1} 1.
$$
Then $w \geq c >0$.
\end{proposition}
\begin{proof} One can write $w$ as
$$
w=1 - (I+(-\Delta)^{-1}V)^{-1} (-\Delta)^{-1} V = 1 - H^{-1} V.
$$
Let $\tilde w = H^{-1} V$. In other words $\tilde w \in C_0 \cap \dot H^1_{loc}$ (in fact, since $V \in \K_0$, $\tilde w$ is continuous) is the unique solution of $H \tilde w = V$ that vanishes at infinity.

We proceed to prove that $\tilde w  \leq 1$. Indeed, let $w_1=(\tilde w-1)_+$. Then multiplying the equation by $w_1$, note that, on the support of $w_1$, $\tilde w = 1 + w_1$, so we obtain
$$
\langle H (1 + w_1), w_1 \rangle = \langle V, w_1 \rangle.
$$
But $H 1 = V$, so this term cancels and we are left with $\langle H w_1, w_1 \rangle =0$. Hence $w_1=0$ almost everywhere, showing that $\tilde w \leq 1$, hence $w \geq 0$.

Since $\lim_{x \to \infty} 1 = 1$ and $\lim_{x \to \infty} H^{-1} V = 0$, it follows that $\lim_{x \to \infty} w = 1$. At the same time, $w$ satisfies the equation $H w = 0$, so by the strong maximum principle Proposition \ref{strongmax} we obtain that $w$ is bounded away from $0$.
\end{proof}

If $H$ has negative energy bound states, then these are orthogonal to $w$. In particular, $w$ being orthogonal to the positive ground state (see Lemma \ref{ground}) means that it cannot have constant sign.
\begin{lemma}\lb{posit} Consider any eigenfunction $f_n$ of $H$, such that $Hf_n = \lambda_n f_n$, $\lambda_n \ne 0$. Then $\langle w, f_j \rangle = 0$. Hence $P_p w = 0$ and $w=P_c w$.
\end{lemma}
\begin{proof}
$$
\langle w_0, f_j \rangle = \lambda_j^{-1} \langle w_0, H f_j \rangle = \lambda_j^{-1} \langle H w_0, f_j \rangle = 0
$$
since $H w_0 = (-\Delta) 1 = 0$.
\end{proof}

On the other hand, by the same token, there always exists a linear combination of $w$ and the ground state with constant sign.
\begin{lemma}\lb{wa} Let $w=(I+(-\Delta)^{-1}V)^{-1}1$ and let $f_N$ be the ground state of $H$ from Lemma \ref{ground}. Assuming that $\lambda_N<0$, then $w_a=w+a f_N>0$ for some $a \in \R$.
\end{lemma}
\begin{proof} Again let $w=1-\tilde w$, $\tilde w=H^{-1}V$, and $w_1=(\tilde w-1)_+$. As $w_1$ is continuous and has compact support, it has a finite maximum, while the ground state $f_N$ has a non-negative minimum on $\supp w_1$. Therefore, for some $a>0$, $w_1 < a f_N$, so $\tilde w < 1+a f_N$ and $w+a f_N > 0$.
\end{proof}



\section{Mihlin multipliers}

Here we prove a complete list of bounds for Mihlin spectral multipliers $m(\sqrt H)$, including weak $(1, 1)$ type and Hardy space endpoint bounds. This is the main result of the paper. The starting point is estimates (\ref{fund}) and (\ref{funda}).

The proof of Theorem \ref{mainthm} is partly based on the following fundamental result about weak $(1, 1)$ and $L^p$ boundedness:
\begin{theorem}\lb{fund_thm}[Theorem 3, p.~19 in \cite{stein}, also see Calder\'{o}n--Zygmund \cite{cazy}]\lb{czthm} Consider an operator $T$ with integral kernel $T(x, y)$ satisfying the cancellation condition
\be\lb{cancel}
\int_{|x-y| \geq c|y-\tilde y|} |T(x, y)-T(x, \tilde y)| \dd x \leq C < \infty.
\ee
If in addition $T$ is $L^2$-bounded, then $T$ is of weak $(1, 1)$ type, bounded from $\H$ to $L^1$, and $L^p$-bounded for $1<p \leq 2$.
\end{theorem}

The cancellation condition (\ref{cancel}) follows trivially if $T \in L^\infty_y L^1_x$ (in which case $T$ is $L^1$-bounded, not a singular integral operator) and less trivially if
\be\lb{grad_cond}
|\nabla_y T(x, y)| \les |x-y|^{-d-1},
\ee
or more generally if, for example, the following H\"{o}lder-type condition holds: whenever $|x-y|>c|\tilde y-y|$
$$
|T(x, y)-T(x, \tilde y)| \les \frac {|\tilde y-y|^\epsilon}{|x-y|^{d+\epsilon}}
$$
for some $\epsilon \in (0, 1]$. In the proof of Theorem \ref{mainthm} we have to use condition (\ref{cancel}) itself.

We also directly use the Calder\'{o}n--Zygmund decomposition:
\begin{theorem}[Theorem 2, p.~17 in \cite{stein}]\lb{czlemma} Suppose we are given a doubling measure $\mu$ on $\R^d$ and a function $f \in L^1(\mu)$, with $\mu(\R^d)=+\infty$, and $\alpha>0$. Then there exist a decomposition of $f$, $f=g+b$, with $b=\sum_n b_n$, and a sequence of balls $B_n=B(y_n, r_n)$, such that: $|g(x)| \les \alpha$ for a.e.~$x$, $$
\supp b_n \subset B_n,\ \int_{\R^d} |b_n| \dd \mu \les \alpha \mu(B_n),\
\int_{\R^n} b_n \dd \mu=0, \text{ and } \sum \mu(B_n) \les \alpha^{-1} \|f\|_{L^1(\mu)},
$$
where all constants only depend on the dimension and on the doubling constant of $\mu$.
\end{theorem}

\begin{proof}[Proof of Theorem~\ref{mainthm}]
By linearity, it suffices to prove the statements when $M_s = 1$ with a bound
that does not depend on $m$ in any other way.

Since $m$ is bounded and $H$ is self-adjoint, $m(H)$ is $L^2$-bounded.

The proof will be presented as a series of lemmas. For convenience, let
$$
T=m(\sqrt H),\ \tilde T=m(\sqrt H)H(-\Delta)^{-1},
$$
which will play an important part in the proof. The main difficulty is that, as we shall see, only $T$ is $L^2$-bounded, but only $\tilde T$ fulfills the cancellation condition (\ref{cancel}). Thus the Calder\'{o}n--Zygmund Theorem \ref{czthm} does not apply directly.

If $V \in \K_0$, the operator $H(-\Delta)^{-1} = I+V(-\Delta)^{-1}$ is $L^1$-bounded, but need not be $L^p$-bounded for any $p>1$ and is definitely not $L^2$-bounded; on the other hand, $H(-\Delta)^{-1}$ is $\dot H^{-1}$-bounded and same for $T$ and $\tilde T$.

To begin with, instead of H\"{o}rmander's condition (\ref{ch}), assume that $m$ satisfies Mihlin's stronger condition (\ref{cm}). This condition simplifies the proof and is satisfied in some important cases where the multipliers have much more than minimum regularity: $m(\lambda)=\lambda^{i\sigma}$ used in complex interpolation, the multipliers appearing in the definition of Paley-Wiener projections or in the proof of the boundedness of the square function, and many others. Afterward we shall also present a proof assuming only H\"{o}rmander's condition (\ref{ch}).

Let $\tilde m(\lambda)=m(|\lambda|)$ be the even extension of $m$ to $\R$. To characterize the Fourier transforms of $m$ and $\tilde m$ under either condition, consider a smooth partition of unity $\sum_{n \in \Z} \phi(2^{-n} \lambda)=\chi_{(0, \infty)}(\lambda)$, where $\phi \in C^\infty$ and $\supp \phi \subset [1/2, 4]$ and let $m_n(\lambda)=\phi(2^{-n} \lambda) m(\lambda)$, $\tilde m_n(\lambda)=\phi(2^{-n} |\lambda|) \tilde m(\lambda)$.

Under Mihlin's condition, for $k, \ell \in \Z$ and $k \geq 0$, $|\partial_\lambda^k [\lambda^\ell m_n(\lambda)]| \les \lambda^{\ell-k}$ and it is supported on $[2^{n-1}, 2^{n+2}]$. Hence
\be\lb{mimihlin}
|t^k \partial_t^\ell \widehat {m_n}(t)| \les 2^{(\ell-k+1)n}
\ee
and each such expression is smooth and vanishes at infinity. In particular, since $m_n(\lambda)$ are supported away from $0$, for each $n$ $\widehat {m_n}(\lambda)$ has infinitely many antiderivatives fulfilling this bound. Also, interpolating between integer values of $k$ proves (\ref{mimihlin}) for all $k \geq 0$.

Consider $t \in [2^{n_0}, 2^{n_0+1}]$. Then for $\ell \geq 0$
\be\lb{sp1}
\sum_{n < -n_0} |\partial_t^\ell \widehat {m_n}(t)| \les \sum_{n < -n_0} 2^{(\ell+1)n} = 2^{-(\ell+1)n_0} \text{ and }
\sum_{n \geq -n_0} t^{\ell+2} |\partial_t^\ell \widehat {m_n}(t)| \les \sum_{n \geq -n_0} 2^{-n} = 2^{1+n_0},
\ee
so $|\partial^\ell_t \widehat m(t)| \les t^{-1-\ell}$ and same for $\tilde m$.

A further improvement is possible for $\tilde m$, because $\tilde m$ is even (so $\widehat {\tilde m}$ corresponds to the cosine transform of $m$), hence so are the dyadic pieces $\tilde m_n$. For each odd $\ell$,
$$
\int_\R \lambda^\ell \tilde m_n(\lambda) \dd \lambda = 0.
$$
Thus each $\lambda^\ell \tilde m_n$ has a bounded antiderivative of size $2^{(\ell+1)n}$ supported on $[-2^{n+2}, 2^{n+2}]$, so (\ref{mimihlin}) holds for $k=-1$ and odd $\ell$ and, in particular, $|t^{-1} \partial_t^{-1} \widehat {\tilde {m_n}}(t)| \les 2^n$.

But then, again considering $t \in [2^{n_0}, 2^{n_0+1}]$,
\be\lb{sp2}
\sum_{n<-n_0} t^{-1} |\partial_t^{-1} \widehat {\tilde {m_n}}(t)| \les \sum_{n < -n_0} 2^{n} = 2^{-n_0} \text{ and }
\sum_{n \geq -n_0} t |\partial_t^{-1} \widehat {\tilde{m_n}}(t)| \les \sum_{n \geq -n_0} 2^{-n} = 2^{1+n_0},
\ee
so $|\partial_t^{-1} \widehat {\tilde m}(t)| \les 1$, meaning that $\widehat {\tilde m}$ admits a bounded antiderivative.

\begin{lemma}\lb{4} If $m$ satisfies Mihlin's condition (\ref{cm}), then
$$
|\nabla_y \tilde T(x, y)| = |\nabla_y m(\sqrt H) H (-\Delta)^{-1}(x, y)| \les |x-y|^{-4}.
$$
\end{lemma}
\begin{proof}
Recall the notations $S_t:=\chi_{t \geq 0} \frac {\sin(t \sqrt H) P_c}{\sqrt H}$ and $C_t:=\chi_{t \geq 0} \frac {\cos(t \sqrt H) P_c}H$, such that $C_t = \int_t^\infty S_\tau \dd \tau$, and likewise for $S_{0t}$ and $C_{0t}$.

Using Stone's formula for the spectral measure of $H$ and then Plancherel's formula, we write
\be\lb{mihlin1}\begin{aligned}
T:=m(\sqrt{H})P_c &= (\pi i)^{-1}\int_{-\infty}^\infty \lambda \tilde m(\lambda)R_V^+(\lambda^2)P_c \,d\lambda = \frac 1 {2\pi^2} \int_0^\infty \partial_t \widehat {\tilde m}(t) S_t \dd t,
\end{aligned}\ee
where recall $\tilde m(\lambda)=m(|\lambda|)$ is the even extension of $m$ to $\R$.

To avail ourselves of Lemma \ref{lemma21}'s estimates in (\ref{mihlin1}), we perform an integration by parts and apply $H(-\Delta)^{-1}$ on the right.

Since the low $t$ regime is different, we use a smooth cutoff (to avoid boundary terms) to separate the two parts. Let $\chi \in C^\infty(\R)$ be such that $\chi(t)=1$ for $t \leq \frac 1 2$ and $\chi(t)=0$ for $t \geq \frac 3 4$. Then
\be\lb{t1t2}\begin{aligned}
T(x, y)&= \frac 1 {2\pi^2} \int_{\frac {|x-y|}2}^\infty (1-\chi({\textstyle \frac t{|x-y|}})) \partial_t \widehat {\tilde m}(t) S_t \dd t + \frac 1 {2\pi^2} \int_0^{\frac {3|x-y|} 4} \chi({\textstyle \frac t{|x-y|}}) \partial_t \widehat {\tilde m}(t) S_t \dd t \\
&=: T_1+T_2.
\end{aligned}\ee
In the high $t$ portion, we express the integral as a function of $C_t$:
\be\lb{t1ct}
T_1(x, y)= -\frac 1 {2\pi^2} \int_{\frac {|x-y|}2}^\infty \partial_t [(1-\chi({\textstyle \frac t{|x-y|}})) \partial_t \widehat {\tilde m}(t)] C_t(x, y) \dd t.
\ee
Hence, noting that
$$
\textstyle \partial_t \chi({\textstyle t \over \alpha}) = \alpha^{-1} \chi'({\textstyle t \over \alpha}) \les t^{-1} \chi'({\textstyle t \over \alpha}) \les \one_{\geq \alpha/2}(t) t^{-1},
$$
we obtain
\be\lb{grbd}\begin{aligned}
|\nabla_y T_1 H (-\Delta)^{-1}(x, y)| &\les \int_{\frac {|x-y|} 2}^\infty t^{-3} |\nabla_y C_t H (-\Delta)^{-1}(x, y)| \dd t \\
&\les \sup_{t \geq \frac {|x-y|} 2} t^{-4} \cdot \int_0^\infty t |\nabla_y C_t H (-\Delta)^{-1}(x, y)| \dd t \les \frac 1 {|x-y|^4}.
\end{aligned}\ee

Without integrating by parts we get
$$
|T_1(x, y)| \les \int_{\frac {|x-y|}2}^\infty (1-\chi({\textstyle \frac t{|x-y|}})) t^{-1} \partial_t \widehat {\tilde m}(t) [t S_t(x, y)] \dd t \les \sup_{t \geq \frac {|x-y|} 2} |t^{-1} \partial_t \widehat {\tilde m}(t)| \les |x-y|^{-3}.
$$

As an aside, for any $x_0 \in \R$, using Lemma \ref{newversion}
\be\lb{kbound}\begin{aligned}
\int_{|x-y| \geq r} \frac 1 {|x-x_0|} |T_1(x, y)| \dd x \les \int_{r}^\infty |\partial_t \widehat {\tilde m}(t)| \int_{\R^3} \frac {|S_t(x, y)|}{|x-x_0|} \dd x \dd t \les r^{-1}.
\end{aligned}\ee

For the other portion, where $t < |x-y|$, we know the explicit form of $S_t$ over the domain of integration (\ref{outside}):
$$
\one_{> t}(|x-y|) S_t(x, y) = -\one_{>t}(|x-y|) \sum_{j=1}^J \frac {\sinh(t\mu_j)}{\mu_j} f_j(x) \ov{f_j}(y).
$$

Thus $T_2$ is not a singular integral operator, and in the absence of negative eigenvalues $T_2=0$. We prove directly that $T_2 \in L^\infty_y L^1_x$, which implies the same for $T_2 H (-\Delta)^{-1}$, since $H (-\Delta)^{-1} = I + V(-\Delta)^{-1} \in \B(L^1)$ when $V \in \K_0$.

Here as well we integrate by parts, but in the opposite direction. The integral is understood to be improper at $0$:
$$
T_2(x, y)=\frac 1 {2\pi^2} \bigg(\lim_{t \to 0} \widehat {\tilde m}(t) S_t(x, y) - \int_0^{\frac {3|x-y|} 4} \widehat {\tilde m}(t) \partial_t[\chi({\textstyle \frac t{|x-y|}}) S_t(x, y)] \dd t \bigg).
$$
Since $\widehat {\tilde m}(t) \les t^{-1}$, the first expression is bounded in $L^\infty_y L^1_x$ by
$$
\sum_{j=1}^J \|f_j(x)\|_{L^1} \|f_j(y)\|_{L^\infty} < \infty.
$$
In the second expression, one more integration by parts results in
\be\lb{tmp}
\partial_t^{-1} \widehat {\tilde m}(t) \nabla_y \partial_t S_t(x, y) \mid_{t=0} - \int_0^{\frac {3|x-y|} 4} \partial_t^{-1} \widehat {\tilde m}(t) \partial^2_t [\chi({\textstyle \frac t{|x-y|}}) \nabla_y S_t(x, y)] \dd t.
\ee
Since $\partial_t^{-1} \widehat {\tilde m} \les 1$, the first term in (\ref{tmp}) is bounded by the same expression as above, while the second term in (\ref{tmp}) can be estimated by
$$
\bigg\|\sum_{j=1}^J \cosh({\textstyle\frac 3 4} |x-y|\mu_j) |f_j(x)| |f_j(y)|\bigg\|_{L^\infty_y L^1_x} \les \sum_{j=1}^J \bigg\|\frac {e^{-\frac 1 4 \mu_j|x|}}{\langle x \rangle}\bigg\|_{L^1} \bigg\|\frac {e^{-\frac 1 4 \mu_j|y|}}{\langle y \rangle}\bigg\|_{L^\infty} < \infty.
$$

A similar explicit computation shows that both $\nabla_y T_2H(-\Delta)^{-1}(x, y)$ and $\nabla_y T_2(x, y)$ are bounded by $|x-y|^{-4}$ and that $T_2$ is also bounded on $\K$.

Consequently, at least when $m$ fulfills the stronger Mihlin condition (\ref{cm}),
$$
|\nabla_y TH(-\Delta)^{-1}(x, y)| \les |x-y|^{-4}
$$
hence its integral kernel has the desired cancellation property (\ref{cancel}).

Finally, $|\tilde T|, \tilde T \in \B(\K)$ because the intertwining operator $H(-\Delta)^{-1} = I + V (-\Delta)^{-1}$ and its integral kernel in absolute value are bounded on $\K$ as well.
\end{proof}

Using this result, we can prove that $T$ is of weak $(1, 1)$ type and $L^p$-bounded for $1<p<\infty$.
\begin{lemma}\lb{lemma44} If $m$ satisfies Mihlin's condition (\ref{cm}), $T$ is of weak $(1, 1)$ type and $L^p$-bounded for $1<p<\infty$.
\end{lemma}
\begin{proof} Our aim is to prove that for any $f \in L^1$ and $\alpha>0$
$$
\mu(\{x \mid |Tf(x)| > \alpha\}) \les \alpha^{-1} \|f\|_{L^1}.
$$
Here $\mu$ is the usual Lebesgue measure on $\R^3$.

Given $f \in L^1$ and a threshold $\alpha$, we perform a Calderon--Zygmund decomposition for $\tilde f = (I+V(-\Delta)^{-1})^{-1}f$, where $\|\tilde f\|_{L^1} \les \|f\|_{L^1}$. Let $\tilde f = \tilde g + \tilde b$, where $\|\tilde g\|_{L^1} \leq \|\tilde f\|_{L^1}$, $\|g\|_{L^\infty} \leq \alpha$, and $b=\sum_{n \geq 1} b_n$ is such that
$$
\int_{\R^3} b_n = 0,\ \supp b_n \subset B_n:=B(y_n, r_n),\ \int_{\R^3} |b_n| \les \alpha \mu(B_n) \les \alpha r_n^3,\ \sum_{n \geq 1} \mu(B_n) \les \alpha^{-1} \|\tilde f\|_{L^1}.
$$
Then
\be\lb{decomp}
f=(I+V(-\Delta)^{-1}) \tilde b + \tilde g + V (-\Delta)^{-1} \tilde g.
\ee
Applying $T$ to each term in (\ref{decomp}), first note that
$$
T (I+V(-\Delta)^{-1}) \tilde b_n = T H (-\Delta)^{-1} \tilde b_n = \tilde T \tilde b_n
$$
and since $\tilde T$ has the cancellation property (\ref{cancel}) by Lemma \ref{4}
$$
\|T (I+V(-\Delta)^{-1}) \tilde b_n\|_{L^1({}^c B(y_n, 2r_n))} \les \|\tilde b_n\|_{L^1},\ \|T (I+V(-\Delta)^{-1}) \tilde b\|_{L^1({}^c \bigcup_n B(y_n, 2r_n))} \les \|\tilde f\|_{L^1}.
$$
What happens inside $B(y_n, 2r_n)$ is irrelevant because $\mu(\bigcup_n B(y_n, 2r_n)) \leq \sum_n \mu(B(y_n, 2r_n)) \les \alpha^{-1} \|\tilde f\|_{L^1}$, so we can include these regions with no prejudice.

Regarding $g$, since $T \in \B(L^2)$
$$
\|Tg\|_{L^2}^2 \les \|g\|_{L^2}^2 \les \alpha \|f\|_{L^1}
$$
and by Chebyshev's inequality
\be\lb{thusly}
\mu(\{x \mid |Tg(x)| > \alpha \}) \les \alpha^{-2} \|Tg\|_{L^2}^2 \les \alpha^{-1} \|f\|_{L^1}.
\ee

We are left with the last term $V(-\Delta)\tilde g$ from (\ref{decomp}), such that $\|V (-\Delta)^{-1} \tilde g\|_{L^1} \les \|V\|_\K \|f\|_{L^1}$. If $\|V\|_\K$ is sufficiently small we can iterate at this point, since a geometric progression is summable in a quasinormed space such as $L^{1, \infty}$ if it decays sufficiently fast. More generally, write $V=V_1+V_2$, where $V_1 \in C^\infty_c$ is smooth and has compact support and $\|V_2\|_\K < \epsilon \|V\|_\K << 1$, and further decompose
$$
(-\Delta)^{-1} = \frac {\chi_{|x-y| \leq 1/\epsilon_1}} {4\pi |x-y|} + \frac {\chi_{|x-y| > 1/\epsilon_1}} {4\pi |x-y|},
$$
where $\epsilon_1$ will be chosen as a function of $\epsilon$. Then
$$
\|V_1 \frac {\chi_{|x-y| > 1/\epsilon_1}} {4\pi |x-y|} \tilde g\|_{L^1} \les \|V_1\|_{L^1} \epsilon_1 \|\tilde g\|_{L^1} \les \epsilon_1 \|V_1\|_{L^1} \|f\|_{L^1}
$$
and
\be\lb{contrib}
\|V_1 \frac {\chi_{|x-y| \leq 1/\epsilon_1}} {4\pi |x-y|} \tilde g\|_{L^2} \les \|V_1\|_{L^2} \Big\|\frac {\chi_{|x| \leq 1/\epsilon_1}} {4\pi |x|}\Big\|_{L^2} \|\tilde g\|_{L^2} \les \epsilon_1^{-1/2} \|V_1\|_{L^2} \|\tilde g\|_{L^2}.
\ee
We iterate to handle the contribution of the error terms
$$
V_2(-\Delta)^{-1} \tilde g + V_1 \frac {\chi_{|x-y| > 1/\epsilon_1}} {4\pi |x-y|} \tilde g,
$$
while (\ref{contrib}) is $L^2$-bounded and can be handled as in (\ref{thusly}).

The iterated error terms are handled as follows: let $E_0=f$ and
$$
E_{n+1}=V_2(-\Delta)^{-1} \tilde g(E_n) + V_1 \frac {\chi_{|x-y| > 1/\epsilon_1}} {4\pi |x-y|} \tilde g(E_n),
$$
where $\tilde g(E_n)$ is the ``good'' term in the Calderon--Zygmund decomposition of $(I+V(-\Delta)^{-1})^{-1} E_n$. Then
\be\lb{iter}
\mu(\{x \mid |Tf(x)|>\alpha\}) \leq \sum_{n \geq 0} \mu(\{x \mid |T(E_n-E_{n+1})(x)|>2^{-n-1}\alpha\}).
\ee
If $\|E_{n+1}\|_{L^1} \leq C_0 \|E_n\|_{L^1}$ for each $n$, with $C_0<1/2$, then the right-hand side is summable. We ensure this through an appropriate choice of $V_1$, $V_2$, and $\epsilon_1$.
	
Consequently $T$ is of weak $(1, 1)$ type and, being also bounded on $L^2$, it must also be bounded on $L^p$ for $1<p \leq 2$ and by duality also for $2 \leq p < \infty$.
\end{proof}

In general, when $m$ only fulfills the weaker H\"{o}rmander condition (\ref{ch}), we shall rely more heavily on its dyadic decomposition and prove a weaker bound for the integral kernel.

\begin{lemma} If $m$ satisfies the H\"{o}rmander condition (\ref{ch}) for some $s>\frac 3 2$, then $T$ is of weak $(1, 1)$ type and $L^p$-bounded for $1<p<\infty$.
\end{lemma}

\begin{proof}
If $m$ fulfills this condition for some $s>3/2$, then $\lambda m'(\lambda)$ fulfills the same condition for $s-1>1/2$, so in particular $m'(\lambda) \les \lambda^{-1}$.

We also retain the following estimates on $\widehat m$ under these weaker assumptions: for $0 \leq k \leq s$ and $\ell \in \Z$
\be\lb{newest}
|t^k \partial_t^\ell \widehat{m_n}(t)| \les 2^{(\ell-k+1)n},\ \|t^k \partial_t^\ell \widehat{m_n}(t)\|_{L^2} \les 2^{(\ell-k+\frac 1 2) n},
\ee
and same for $\tilde m_n$, while for $\tilde m_n$ we further have $|t^{-1} \partial_t^\ell \widehat {\tilde {m_n}}| \les 2^{(\ell+2)n}$ for odd $\ell$.

We use the same split as in (\ref{t1t2}), $T=T_1+T_2$, but further decompose $T_1$ into the terms $T_{n, 1}$ corresponding to each dyadic piece $\tilde m_n$:
$$
T_1= \sum_{n \in \Z} \frac 1 {2\pi^2} \int_{\frac {|x-y|}2}^\infty (1-\chi({\textstyle \frac t{|x-y|}})) \partial_t \widehat {\tilde {m_n}}(t) S_t \dd t=:\sum_{n \in \Z} T_{n, 1}.
$$

The estimates
$$
|\widehat {m_n}(t)| \les 2^n,\ |t^s \widehat {m_n}(t)| \leq 2^{(1-s)n},\ |t^{-1} \partial_t^{-1} \widehat {\tilde {m_n}}(t)| \les 2^n,\ |t \partial_t^{-1} \widehat {\tilde {m_n}}(t)| \les 2^{-n}
$$
still hold. Since $s>1$, the first two estimates imply that $\widehat {\tilde m}(t) \les t^{-1}$, as in (\ref{sp1}). The last two again imply that $|\partial_t^{-1} \widehat {\tilde m}(t)| \les 1$, as in (\ref{sp2}). Hence $T_2$ can be evaluated in the same manner as above.

The fact that $\|\chi_{|x-y|>r} T_1(x, y)\|_{\B(L^1, \K)} \les r^{-1}$ admits an equally simple proof. Following (\ref{kbound}), we need to establish that under the weaker H\"{o}rmander condition (\ref{ch}), for any $t_0>0$,
$$
\int_{t_0}^\infty |\partial_t \widehat m(t)| \dd t \les t_0^{-1}.
$$
With no loss of generality, set $t_0=2^{n_0}$. From (\ref{newest}) one has in particular that
$$
|t^s \partial_t \widehat{m_n}(t)| \les 2^{(2-s)n},\ \|t^s \partial_t \widehat{m_n}(t)\|_{L^2} \les 2^{(3/2-s) n}.
$$
Take $s<2$. Since $s>1$, for $n < -n_0$ we use
$$
\int_{2^{n_0}}^\infty |\partial_t \widehat {m_n}(t)| \dd t \les 2^{(2-s)n} \int_{2^{n_0}}^\infty t^{-s} \dd t = 2^{(2-s)n + n_0(1-s)}.
$$
Summing all such terms with $n<-n_0$ yields a contribution of size $2^{-n_0}$.

For $n \geq -n_0$, since $s>3/2$, we use
$$
\int_{2^{n_0}}^\infty |\partial_t \widehat {m_n}(t)| \dd t \leq \|t^{-s}\|_{L^2(2^{n_0}, \infty)} \|t^s \partial_t \widehat{m_n}(t)\|_{L^2} \les 2^{n_0(1/2-s)+(3/2-s)n}.
$$
Upon summing in $n$ this again produces a contribution of size $2^{-n_0}$. Hence $|T_1| \in \B(\K)$.


To prove the cancellation property (\ref{cancel}) for $T_1$, we estimate each term $T_{n, 1}$ and its gradient. Suppose $|x-y| \in [2^{n_0}, 2^{n_0+1}]$. Then by Lemma~\ref{newversion}
$$\begin{aligned}
\|T_{n,1}(x, y)\|_{L^1_x({}^cB(y, 2^{n_0}))} &\les \|\sup_{\alpha \geq 2^{n_0}} (1-\chi({\textstyle \frac t \alpha})) t \partial_t \widehat {\tilde {m_n}}\|_{L^1} \les \|\chi_{\geq 2^{n_0-1}}(t) t \partial_t \widehat {\tilde {m_n}}\|_{L^1} \\
&\les \|t^s \partial_t \widehat {\tilde {m_n}}\|_{L^2} \|\chi_{\geq 2^{n_0-1}}(t) t^{1-s}\|_{L^2} \les 2^{(\frac 3 2-s) n} 2^{(\frac 3 2-s)n_0}.
\end{aligned}$$
Since $s>\frac 3 2$, we can sum in $n_0$ over the complement of a ball centered at $y$ of radius $r$ and obtain
\be\lb{compl}
\|T_{n, 1}(x, y)\|_{L^1_x({}^c B(y, r))} \les 2^{(\frac 3 2-s) n} r^{\frac 3 2-s}.
\ee


Likewise, by (\ref{eq:conicalbound2}) and (\ref{t1ct}), recalling that $\partial_t \chi({\textstyle t \over \alpha}) \les \one_{\geq \alpha/2}(t) t^{-1}$, we obtain
$$\begin{aligned}
\|\nabla_y T_{n, 1} H (-\Delta)^{-1}&(x, y)\|_{L^1_x(B(y, 2^{n_0+1})\setminus B(y, 2^{n_0}))} \les\\
&\les \|\sup_{\alpha \geq 2^{n_0}} t \partial_t [(1-\chi({\textstyle \frac t \alpha})) \partial_t \widehat {\tilde {m_n}}(t)]\|_{L^1} \\
&\les \|\one_{\geq 2^{n_0-1}}(t) t \partial_t^2 \widehat {\tilde {m_n}}(t)\|_{L^1} + \|\chi_{\geq 2^{n_0-1}}(t) \partial_t \widehat {\tilde {m_n}}(t)\|_{L^1} \\
&\les (\|t^s \partial_t^2 \widehat {\tilde {m_n}}\|_{L^2} + \|t^{s-1} \partial_t \widehat {\tilde {m_n}}(t)\|_{L^2}) \|\one_{\geq 2^{n_0-1}}(t) t^{1-s}\|_{L^2} \\
&\les 2^{(\frac 5 2 - s)n} 2^{(\frac 3 2-s) n_0}.
\end{aligned}$$


%

Again, summing over the complement of a ball we get that
\be\lb{cancel2}
\|\nabla_y T_{n, 1}H (-\Delta)^{-1}(x, y)\|_{L^1_x({}^cB(y, r))} \les 2^{(\frac 5 2 -s)n} r^{\frac 3 2 - s}.
\ee

Consider $f \in L^1$ and a threshold $\alpha>0$. We again perform a Calderon-Zygmund decomposition of $\tilde f = (I+V(-\Delta)^{-1})^{-1} f$, $\|\tilde f\|_{L^1} \les \|f\|_{L^1}$. Let $\tilde f = \tilde g + \tilde b$, where $\|\tilde g\|_{L^1} \leq \|\tilde f\|_{L^1}$, $\|\tilde g\|_{L^\infty} \leq \alpha$, and $b=\sum_{n \geq 1} b_n$ is such that
$$
\int_{\R^3} b_n = 0,\ \supp b_n \subset B_n:=B(y_n, r_n),\ \int_{\R^3} |b_n| \les \alpha \mu(B_n) \les \alpha r_n^3,\ \sum_{n \geq 1} \mu(B_n) \les \alpha^{-1} \|\tilde f\|_{L^1}.
$$

Then
\be\lb{dec}
T f= T (I + V(-\Delta)^{-1}) g + \sum_{n \geq 1} T (I + V (-\Delta)^{-1}) b_n.
\ee
For the first term in (\ref{dec}) we proceed as in the proof of Lemma \ref{lemma44} following (\ref{thusly}): since $g \in L^2$, $Tg \in L^2$, while $T V (-\Delta)^{-1} g$ is the sum of an $L^2$ term and an error term for which we iterate.

For the second term in (\ref{dec}), for each ``bad'' term $b_n$ we split $T$ into $T_{<-n_0}$ and $T_{\geq -n_0}$, where $n_0=\lfloor \log_2 r_n \rfloor$. Then $T_{<-n_0} H(-\Delta)^{-1}$ has the cancellation property (\ref{cancel2}), ensuring that
$$
\|T_{<-n_0} H(-\Delta)^{-1} b_n\|_{L^1({}^cB(y_n, 2r_n))} \les \|b_n\|_{L^1}.
$$

Likewise, $T_{\geq -n_0}$ has the property (\ref{compl}), so $\|T_{\geq -n_0} b_n\|_{L^1({}^cB(y_n, 2r_n))} \les \|b_n\|_{L^1}$ as well. We are left with $T_{\geq -n_0} V (-\Delta)^{-1} b_n$ from (\ref{dec}). To handle this we increase the outer radius, for each $n$, from $2r_n$ to $2Mr_n$, by some constant factor $M \geq 1$ depending only on $V$, not on $f$ or $n$. We use one extra power of decay gained by
$$
\||y-y_n| V (-\Delta)^{-1} f\|_{L^1} \les \|V\|_\K \||y-y_n|f\|_{L^1}
$$
when $\langle f, 1 \rangle = \int f = 0$, then choose $M$ such that the remainder term, sum over $n$ of
$$
\|\chi_{{}^cB(y_n, Mr_n)} V (-\Delta)^{-1} b_n\|_{L^1} \les (Mr_n)^{-1} \|V\|_\K r_n \|b_n\|_{L^1},
$$
is small in $L^1$ and we can handle it by iterating.

What happens within the radius $2M r_n$ is still irrelevant, since
$$
\mu(\bigcup_{n \geq 1} B(y_n, 2Mr_n)) \les \alpha^{-1} \|\tilde f\|_{L^1}.
$$

At the same time, due to (\ref{compl}), for $n_0=\lfloor \log_2 r_n \rfloor$
$$
\|T_{\geq -n_0} \chi_{B(y_n, M r_n)} V (-\Delta)^{-1} b_n\|_{L^1({}^cB(y_n, 2M r_n))} \les 2^{(\frac 3 2 -s) n_0} (Mr_n)^{\frac 3 2-s} \|b_n\|_{L^1} \les \|b_n\|_{L^1}.
$$
Thus the contributions of all terms are bounded except for the error terms, which we handle by iterating, as in (\ref{iter}).
\end{proof}

We next prove the Hardy space boundedness of a class of Mihlin multipliers. It is convenient to assume that $V \in \K_0 \cap L^{3/2, \infty}$ and use Mihlin's stronger condition (\ref{cm}) instead of (\ref{ch}). While probably not strictly necessary, these stronger assumptions simplify the proof substantially.

\begin{lemma} If $V \in \K_0 \cap L^{3/2, \infty}$ and $m$ satisfies Mihlin's condition (\ref{cm}), $\tilde T=m(\sqrt H) H (-\Delta)^{-1}$ is bounded from the Hardy space $\H$ to $L^1$.
\end{lemma}
\begin{proof} The proof follows that of Theorem 6.3a in \cite[p. 178]{stein}, except that we replace $L^2$ by $L^{3/2, \infty}$.

We use the atomic decomposition of the Hardy space $\H$, following \cite[Theorem 2, p.~107]{stein}: Any function $f \in \H$ can be written as a norm-converging linear combination of $\H$ atoms, $f=\sum_n \lambda_n f_n$, with
$$
\sum_n |\lambda_n| \les \|f\|_{\H}.
$$
Without loss of generality we can then assume $f$ itself is an $\H$ atom: there exists a ball $B=B(y_0, r)$ such that $\supp f \subset B$, $|f| \leq \mu(B)^{-1}$, and $\int_{\R^3} f = 0$.

For $x \in {}^cB(y_0, cr)$, the cancellation properties (\ref{cancel}) of $\tilde T$ and of $f$ itself ensure that
$$\begin{aligned}
& \|\tilde T f\|_{L^1(\R^3 \setminus B(y_0, cr))} = \int_{\R^3 \setminus B(y_0, cr)} \bigg|\int_{\R^3} \tilde T(x, y) f(y) \dd y\bigg| \dd x \\
& \leq \int_{\R^3} \bigg(\int_{\R^3 \setminus B(y_0, cr)} |\tilde T(x, y) - \tilde T(x, y_0)| \dd x\bigg) |f(y)| \dd y \les \|f\|_{L^1} \leq 1.
\end{aligned}$$

For $x \in B(y_0, cr)$, we use that by real interpolation, see \cite{bergh}, $T \in \B(L^{3/2, \infty})$ and $\tilde T \in \B(\K, L^{3/2, \infty})$:
$$
\|\tilde T f\|_{L^1(B(y_0, cr))} \les cr \|\tilde T f\|_{L^{3/2, \infty}} \les cr \|H(-\Delta)^{-1} f\|_{L^{3/2, \infty}} \les cr \|f\|_\K \les r \|f\|_{L^{3/2, 1}} \les 1,
$$
considering that $\|f\|_{L^1} \leq 1$, $\|f\|_{L^\infty} \les r^{-3}$, so by real interpolation, $\|f\|_{L^{3/2, 1}} \les r^{-1}$.
\end{proof}

We can also bound the integral kernel of $m(\sqrt H)$. However, this weak-type bound does not directly imply any further operator bounds.
\begin{lemma} $T=m(\sqrt H) \in L^\infty_y L^{1, \infty}_x$. Furthermore, for any $y \in \R^3$ and $r>0$ consider a set $S$ of measure $r^3$, $S \subset {}^cB(y, r)$. Then
\be\lb{est_noua}
\int_S |T(x, y)| \dd x \les 1.
\ee
\end{lemma}

\begin{proof}
We use the same decomposition (\ref{t1t2}) into $T_1$ and $T_2$. Regarding $T_1$, recall the estimate (\ref{compl})
$$
\|T_{n, 1}(x, y)\|_{L^1_x({}^c B(y, r))} \les 2^{(\frac 3 2-s) n} r^{\frac 3 2-s},
$$
and we also have the $L^\infty$ bound
$$
|T_{n,1}(x, y)| \les \sup_{t \geq \frac {|x-y|} 2} t^{-1} \partial_t \widehat {\tilde {m_n}}(t) \int_0^\infty t S_t(x, y) \dd t \les 2^{3n}.
$$
We want to prove that uniformly in $y$ and for any $\alpha>0$
$$
\mu(\{x \mid |T_1(x, y)| > \alpha\}) \les \alpha^{-1}.
$$
Since $\mu(B(y, \alpha^{-1/3})) \les \alpha^{-1}$, it suffices to consider $T_1$ in the complement of this ball of radius $r=\alpha^{-1/3}$. For $n < n_0$
$$
|T_{<n_0, 1}(x, y)| \leq \sum_{n < n_0} |T_{n, 1}(x, y)| \les 2^{3n_0},
$$
so by taking $n_0=\lfloor \frac {(\log_2 \alpha) -1} 3 \rfloor$, such that $2^{3n_0}<\alpha/2$, we obtain an admissible quantity and $2^{n_0} \sim \alpha^{1/3} = r^{-1}$. Finally,
$$
\|T_{\geq n_0, 1}(x, y)\|_{L^1_x({}^c B(y, r))} \leq \sum_{n \geq n_0} \|T_{n, 1}(x, y)\|_{L^1_x({}^c B(y, r))} \les 2^{(\frac 3 2-s) n_0} r^{\frac 3 2-s} \les 1,
$$
so by Chebyshev's inequality
$$
\mu(\{x \in {}^c B(y, r)\mid |T_{\geq n_0, 1}(x, y)| > \alpha\}) \leq \frac {2\|T_{\geq n_0, 1}(x, y)\|_{L^1_x({}^c B(y, r))}} \alpha \les \alpha^{-1}.
$$
It follows that $T_1(x, y) \in L^{1, \infty}_x$, uniformly in $y$. Regarding $T_2$, one can prove as above that
$$
|T_2(x, y)| \les \sum_{j=1}^J \cosh({\textstyle\frac 3 4} |x-y| \mu_j) |f_j(x)| |f_j(y)| \les \langle x \rangle^{-3} \langle y \rangle^{-3} \sum_{j=1}^J \langle \mu_j^{-4} \rangle \les |x-y|^{-3} \in L^{1, \infty}_x
$$
uniformly in $y$. Since $L^{1, \infty}$ is a quasinormed space, it follows that $T=T_1+T_2 \in L^\infty_y L^{1, \infty}_x$.

The more precise estimate (\ref{est_noua}) uses the same decomposition. With $n_0$ as above, $T_{<n_0, 1}$ is of size $r^{-3}$, so its integral on a set of size $r^3$ will have size $1$. On the other hand, the integral of $T_{\geq n_0, 1}$ in ${}^c B(y, r)$ is also of size $1$.
\end{proof}

As shown in \cite{becgol3}, if $m$ satisfies~\eqref{ch} with the stronger assumption $s > 2$, then $|T(x, y)| \les |x-y|^{-3}$, which is stronger than the above statement.

This concludes the proof of Theorem \ref{mainthm}.
\end{proof}

\section{The twisted Hardy space}
\begin{proof}[Proof of Proposition \ref{twist}] Theorem \ref{mainthm} implies that
$$
\|H^{1+i\sigma} (-\Delta)^{-1}\|_{\B(\H, L^1)} \les C_2=\sup_{\lambda>0} \max(|\lambda^{i\sigma}|, \lambda |(\lambda^{i\sigma})'|, \lambda^2 |(\lambda^{i\sigma})''|) \les \langle \sigma \rangle^2,
$$
The first conclusion $\tilde \H \subset L^1$ follows immediately from the definition of $\tilde \H$.

Next,
$$
\langle H^{1+i\sigma}P_c (-\Delta)^{-1} f, w \rangle = \langle H^{i\sigma}P_c (-\Delta)^{-1} f, Hw \rangle = 0.
$$

The operators $H^{i\sigma}$ are bounded on $\tilde \H$ because they act by translation on $f_\sigma$ in Definition \ref{dhardy}. In turn, translation is a (non-uniformly) bounded operation on $\langle \sigma \rangle^{-s} \M$ and we obtain a bound of
$$
\|H^{i\sigma}\|_{\B(\tilde \H)} \les \langle \sigma \rangle^2.
$$

More general operators can be bounded by representing $m(H)$ as an integrable combination of $H^{i\sigma}$, $\sigma \in \R$. Indeed, a computation
$$
m(t)=\int_\R t^{i\sigma} g(\sigma) \dd \sigma = \int_\R e^{i\sigma\log t} g(\sigma) \dd \sigma
$$
shows that such a representation always exists, with
$$
g(\sigma)=2\pi \mc F^{-1} [m(e^t)](\sigma).
$$
If $m \in C^\infty_c((0, \infty))$, then $g$ satisfies the required integrability condition. Milder decay conditions at $0$ and $\infty$ suffice, but observe that, for example, $\tilde P_{\leq n}$ and $\tilde P_{\geq n}$ are not included.
\end{proof}

\begin{proof}[Proof of Theorem \ref{hor}] By Theorem \ref{mainthm}, for $V \in \K_0 \cap L^{3/2, \infty}$ and $m$ satisfying Mihlin's condition (\ref{cm}), $m(H)HP_c(-\Delta)^{-1}$ is bounded from $\H$ to $L^1$.

We need to obtain the same for $m(H)H^{1+i\sigma}P_c(-\Delta)^{-1}$, plus control on the operator norm growth rate. But observe that the latter expression has the same form as the initial one, with $m(H)$ replaced by $m(H)H^{i\sigma}$. If $m$ fulfills Mihlin's condition (\ref{cm}), so does $m(\lambda)\lambda^{i\sigma}$, with a norm of size $\langle \sigma \rangle^2$.  This polynomial growth rate suffices in light of Definition \ref{dhardy}.
\end{proof}

\begin{proof}[Proof of Lemma \ref{inter}] Let $\tilde L^p = [\tilde \H, L^2]_\theta$, where $\frac {1-\theta} 1 + \frac \theta 2 = \frac 1 p$. Since $\tilde \H \subset L^1$, clearly $\tilde L^p \subset L^p$. We also need to prove the opposite continuous inclusion.

Then $T_{1+i\sigma}=H^{1+i\sigma} (-\Delta)^{-1-i\sigma} = \H^{1+i\sigma} (-\Delta)^{-1} (-\Delta)^{-i\sigma}$ is a family of bounded maps from $\H$ to $\tilde \H$ (since $(-\Delta)^{i\sigma}$ is bounded on $\H$), with a norm of polynomial growth, while $T_{-1/2+i\sigma}=H^{-1/2+i\sigma} (-\Delta)^{1/2-i\sigma}$ are bounded on $L^2$. Hence, by interpolation, the identity map is bounded from $L^{3/2}$ to $\tilde L^{3/2}$.

This proves that $L^{3/2} \subset \tilde L^{3/2}$, so $\tilde L^{3/2}=L^{3/2}$ and in general $\tilde L^p=L^p$ for $3/2\leq p \leq 2$.
\end{proof}

\section{Intertwining operators}
We include the proof of the following result from \cite{goldberg}.
\begin{lemma}\lb{hilbert} Assume that $V \in \K_0$ and $H=-\Delta+V$ has no zero energy eigenstate or resonance. Then the operators $H^sP_c(-\Delta)^{-s}$ and $(-\Delta)^s H^{-s}P_c$ are bounded on $L^2$ for $|s| \leq \frac 1 2$.
\end{lemma}
\begin{proof} It suffices to prove that $H^{1/2} P_c (-\Delta)^{-1/2}$ and $H^{-1/2} P_c (-\Delta)^{1/2}$ are bounded on $L^2$. By a $T T^*$ argument, this reduces to showing that $H \in \B(\dot H^1, \dot H^{-1})$ and $H^{-1} \in \B(\dot H^{-1}, \dot H^1)$.

The first goal further reduces to showing that $V \in \B(\dot H^1, \dot H^{-1})$. Indeed, note that $|V|^{i\sigma} \in \B(L^\infty)$, while $|V|^{1+i\sigma} \in \B((-\Delta)^{-1}L^1, L^1)$. By complex interpolation it follows that $|V|^{1/2} \in \B(\dot H^1, L^2)$ and, by duality, $|V|^{1/2} \in \B(L^2, \dot H^{-1})$ as well. Hence
\be\lb{l2}
|V| \in \B(\dot H^1, \dot H^{-1}).
\ee

The second goal is equivalent to asking that $H^{-1}(-\Delta) = (I+(-\Delta)^{-1}V)^{-1} \in \B(\dot H^1)$. Clearly $(-\Delta)^{-1} V \in \B(\dot H^1)$; furthermore, since $V \in \K_0$, this operator is compact. Then by Fredholm's alternative $I+(-\Delta)^{-1}V$ can be inverted unless the equation
$$
g+V(-\Delta)^{-1} g =0
$$
has a nonzero solution $g \in \dot H^{-1}$. Letting $f=(-\Delta)^{-1} g$, we obtain that $f+(-\Delta)^{-1} V f = 0$, $f \in \dot H^1$. But then by the usual bootstrapping we get that $Vf \in L^1 \cap \K$ and $f \in (-\Delta)^{-1} L^1 \cap L^\infty$ must be a zero energy eigenfunction or resonance of $H$, whose existence is precluded by our hypotheses. Hence $(I+(-\Delta)^{-1}V)^{-1}$ is bounded on $\dot H^1$ as claimed.
\end{proof}

\begin{proof}[Proof of Theorem \ref{opt}] Here we do not assume that $V \in L^{3/2, \infty}$, so we cannot use results about the twisted Hardy space $\tilde \H$.
	
Clearly
$$
H^{1/2+i\sigma} (-\Delta)^{-1/2-i\sigma} \in \B(L^2).
$$

Furthermore, $H^{i\sigma}, (-\Delta)^{-i\sigma} \in \B(L^p)$ implies that
$$
H^{i\sigma} (-\Delta)^{-i\sigma} \in \B(L^p)
$$
and $H^{i\sigma} \in \B(L^1, L^{1, \infty})$ implies that
$$
H^{i\sigma} (-\Delta)^{-i\sigma} \in \B(\H, L^{1, \infty}).
$$
Fixing some $s>3/2$ in Theorem \ref{mainthm}, all operator norms have polynomial growth in $\sigma$.

By complex interpolation, for $0<s<1$ and $\frac 1 {1-s}<p< \frac 1 s$, $H^s(-\Delta)^{-s} \in \B(L^p)$.

For the reversed expression, we cannot obtain the full range of bounds, due to the limitations of Lemma \ref{inter}. However, by complex interpolation between $(-\Delta)^{i\sigma}H^{-i\sigma} \in \B(L^p)$, $1<p<\infty$, and $(-\Delta)^{1/2+i\sigma} H^{-1/2-i\sigma} \in \B(L^2)$ due to Lemma \ref{hilbert}, we obtain that $H^s (-\Delta)^{-s} \in \B(L^p)$ for $\frac 1 {1-s} < p < \frac 1 s$.


All the other claimed bounds follow by duality.
\end{proof}

\begin{proof}[Proof of Lemma \ref{aux}] Recall that
$$
T:=m(\sqrt{H}) = (\pi i)^{-1}\int_{-\infty}^\infty \partial_t \widehat {\tilde m}(t) S_t \dd t,
$$
so by Lemma \ref{newversion}
$$
\|T\|_{\B(L^1)}=\esssup_y \int_{\R^3} |T(x, y)| \dd x \les \int_0^\infty t \partial_t \widehat {\tilde m}(t) \dd t < \infty.
$$
The same is true, symmetrically, for the $\B(L^\infty)$ norm, and for all the others by interpolation.

More smoothness for $m$ implies more decay for the kernel of $m(\sqrt H)$. Due to (\ref{fund}) or (\ref{funda}), if there are no bound states
$$
\int_{-\infty}^\infty \partial_t \widehat {\tilde m}(t) S_t(x, y) \dd t \les \int_{|x-y|}^\infty |\widehat {\tilde m}(t)| |S_t(x, y)| \dd t \les \frac {\sup_{t \geq |x-y|} |\partial_t \widehat {\tilde m}(t)|}{|x-y|} \les \frac {R^2 \|m^{(n)}\|_{L^1}} {|x-y|^{n+1}},
$$
assuming that $\supp m \subset [0, R]$, $m^{(n)} \in L^1$, and all odd derivatives of $m$ of order up to $n$ vanish at $0$.

If there are bound states, we use a cutoff of $\frac {|x-y|} 2$ instead and need to separately bound the contribution of the interval $t \in [0, |x-y|]$, where we use (\ref{outside}). For the contribution of each bound state, we obtain a bound of
$$
\int_0^{\frac {|x-y|} 2} \partial_t \widehat {\tilde m}(t) \frac {\sinh(\sqrt{-\lambda_k} t)}{\sqrt{-\lambda_k}} \dd t \les \sup_t |\partial_t \widehat {\tilde m}(t)| \frac {\sinh(\sqrt{-\lambda_k} |x-y|/2)}{\sqrt{-\lambda_k}}.
$$
Including the omitted factor of
$$
f_k(x) \otimes \ov{f_k(y)} \les \frac {e^{-\sqrt{-\lambda_k}(|x|+|y|)}}{\langle x \rangle \langle y \rangle},
$$
we obtain arbitrary decay here as well. The cutoff can be further adjusted to preserve scaling.
\end{proof}

\begin{proof}[Proof of Proposition \ref{paleywiener}] We start by proving that the twisted Paley--Wiener projections and similar operators are bounded from $L^1$ to $L^\infty$. Denote
$$
S^1_t := \int_0^t C_t \dd t = \frac {\sin(t \sqrt H) P_c}{H^{3/2}}.
$$
In the free case, $S_{0t}^1$ has the explicit integral kernel
$$
S_{0t}^1(x, y) = \left\{
\begin{aligned}
&\frac {\sgn t} {4\pi},& |x-y| \leq |t|\\
&\frac 1 {4\pi} \frac t {|x-y|},& |t| \leq |x-y|,
\end{aligned}\right.
$$
hence is uniformly bounded in $x$, $y$, and $t$ by $\frac 1 {4\pi}$. From the \cite{becgol2} estimate
$$
\bigg\|\frac {\sin(t \sqrt H) P_c}{\sqrt H} f\bigg\|_{L^\infty_{t, x}} \les \|\Delta f\|_{L^1}
$$
we obtain a similar uniform bound in the general case.

For a multiplier $m$, starting from the formula
$$
T:=m(\sqrt{H})P_c = \frac 1 {2\pi^2} \int_0^\infty \partial_t \widehat {\tilde m}(t) S_t \dd t,
$$
we again split the integral in two, $T=T_1+T_2$, using a smooth cutoff at $t \sim |x-y|$, as in the proof of Theorem \ref{mainthm}. Integrating by parts twice in the high frequency part $T_1$, as in (\ref{t1ct}), leads to
$$
T_1(x, y)= \frac 1 {2\pi^2} \int_{\frac {|x-y|}2}^\infty \partial_t^2 [(1-\chi({\textstyle \frac t{|x-y|}})) \partial_t \widehat {\tilde m}(t)] S_t^1(x, y) \dd t.
$$
Since $\chi'$ and $\chi''$ are supported on compact intervals and the latter integrates to $0$,
$$\begin{aligned}
\|\partial_t^2 [(1-\chi({\textstyle \frac t \alpha})) \partial_t \widehat {\tilde m}(t)]\|_{L^1} &\les \|\partial_t^2 \widehat {\tilde m}\|_{L^1} + \|\alpha^{-1} \chi'({\textstyle \frac t \alpha}) \partial_t \widehat {\tilde m}\|_{L^1} + \|\alpha^{-2} \chi''({\textstyle \frac t \alpha}) \widehat {\tilde m}\|_{L^1} \\
&\les \|\partial_t^2 \widehat {\tilde m}\|_{L^1} + \|\partial_t \widehat {\tilde m}\|_{L^\infty} \les \|\partial_t^2 \widehat {\tilde m}\|_{L^1}.
\end{aligned}$$
Hence
$$
\|T_1\|_{\B(L^1, L^\infty)} \les \|\partial_t^2 \widehat {\tilde m}\|_{L^1}.
$$

For the low frequency component $T_2$,
$$
T_2=\frac 1 {2\pi^2} \int_0^{\frac {3|x-y|} 4} \chi({\textstyle \frac t{|x-y|}}) \partial_t \widehat {\tilde m}(t) S_t \dd t,
$$
we do not need to integrate by parts, because in this region, outside the light cone, $S_t$ has good decay and smoothness properties, as per (\ref{outside}). Integrating on this interval of length $|x-y|$, we obtain a bound of
$$
T_2 \les \|\partial_t \widehat {\tilde m}\|_{L^\infty} |x-y| \sum_{n=1}^N \frac {\sinh(3\sqrt{-\lambda_n}|x-y|/4)}{\sqrt{-\lambda_n}} f_n((x) \ov f_n(y) \les \|\partial_t \widehat {\tilde m}\|_{L^\infty} \les \|\partial_t^2 \widehat {\tilde m}\|_{L^1}.
$$

Assuming, for example, that $m \in C^\infty_c([0, \infty))$ (a milder condition, such as $\langle \lambda \rangle^s m(\lambda) \in L^2$, $s>\frac 5 2$, would suffice), this norm is finite and it rescales like $\alpha^2$ when replacing $m$ by $m(\alpha^{-1} \lambda)$:
$$
\partial_t^2 \mc F [{\tilde m}(\alpha^{-1} \lambda)](t) = \alpha^3 \partial_t^2 \widehat {\tilde m}(\alpha t).
$$

This proves the $L^1$ to $L^\infty$ bound, which by real interpolation with Lemma \ref{aux} produces $L^1$ to $L^p$ bounds, $1 \leq p \leq \infty$, and then all the others by complex interpolation (for a fixed operator) and duality.

Finally, this Fourier-side norm is translation-invariant, so for example
$$
\|e^{i\tau \sqrt H} P_n\|_{\B(L^1, L^\infty)} \les \|2^{3n} \partial_t^2 \widehat {\tilde \phi}(2^n t)\|_{L^1} \les 2^{2n}.
$$
In other words, the norm is bounded independently of $\tau$.
\end{proof}

\section{The wave equation} \lb{sec:Strichartz}


Recalling the discussion from \cite{becgol3}, for the free wave equation in $\R^{d+1}$
$$
u_{tt}-\Delta u = F,\ u(0)=u_0,\ u_t(0)=u_1,
$$
as proved by \cite{keeltao}, the following Strichartz estimates hold:
\be\lb{strichartz_est}
\|u\|_{L^\infty_t \dot H^s_x \cap L^p_t L^q_x} + \|u_t\|_{L^\infty_t \dot H^{s-1}_x} \les \|u_0\|_{\dot H^s} + \|u_1\|_{\dot H^{s-1}} + \|F\|_{L^{\tilde p'}_t L^{\tilde q'}_x},
\ee
where by scaling and translation invariance
\be\lb{condition1}
\frac 1 p + \frac d q = \frac d 2 - s = \frac 1 {\tilde p'} + \frac d {\tilde q'} - 2,\ 2 \leq p, q, \tilde p, \tilde q \leq \infty,
\ee
and the exponents must be wave-admissible:
\be\lb{condition2}
\frac 2 p + \frac {d-1} q \leq \frac {d-1} 2,\ \frac 2 {\tilde p} + \frac {d-1} {\tilde q} \leq \frac {d-1} 2.
\ee
The endpoint cases $(p, q) = (2, \infty)$ in $\R^{3+1}$ and $(p, q) = (4, \infty)$ in $\R^{2+1}$ are not true (and same for the inhomogeneous term, i.e.\ $(\tilde p, \tilde q) = (2, \infty)$ in $\R^3$).

In three dimensions, for wave-admissible exponents $(p, q)$, the pairs $(\frac 1 p, \frac 1 q)$ cover the triangle with vertices $(0, 0)$, $(0, \frac 1 2)$, and $(\frac 1 2, 0)$.

Thus, in $\R^{3+1}$, the homogeneous Strichartz estimates (i.e.\;$F=0$)
\be\lb{hom_str}
\|u\|_{L^p_t L^q_x} \les \|u_0\|_{\dot H^s} + \|u_1\|_{\dot H^{s-1}}
\ee
follow by interpolation between the segment $(\infty, \frac 6 {3-2s})$, $0 \leq s < 3/2$ --- trivial by $\dot H^s$ norm conservation --- and the sharp admissible segment $\frac 2 p + \frac {d-1} q = \frac {d-1} 2$, i.e.
\be\lb{sharp_str}
\frac 1 p + \frac 1 q = \frac 1 2.
\ee

We now proceed with the proof of Lemma \ref{lplemma}, which plays a crucial part in the proof of Strichartz estimates. For comparison and later use, we include the following auxiliary result from \cite{becgol3}:

\begin{lemma}\lb{lm} For $s_1, s_2 \geq 0$, $0 \leq s = s_1 + s_2 < 1$,
\be\lb{Delta}
\|(-\Delta)^{-s_1} \cos(t \sqrt H) P_c (-\Delta)^{-s_2} f\|_{L^{\frac 2 {1-s}}} \les |t|^{-s} \|f\|_{L^{\frac 2 {1+s}}}.
\ee
\end{lemma}

\begin{proof}[Proof of Lemma \ref{lplemma}] Starting from the estimate
$$
\|C_t f\|_{L^\infty} \les |t|^{-1} 
$$
and using the fact that by Lemma \ref{aux} $\|\tilde P_n H^{i\sigma}\|_{\B(L^p)}$ is bounded independently of $n$ and grows polynomially with $\sigma$, we obtain
$$
\|\cos(t \sqrt H) H^{-1-i\sigma} \tilde P_n f\|_{L^\infty} \les |t|^{-1} \langle \sigma \rangle^{\frac 3 2 +} \|f\|_{L^1}.
$$
Interpolating with the $L^2$ bound, we obtain that for $1 \leq p \leq 2$ and $s=\frac 2 p-1$, independently of $n$
$$
\|\cos(t \sqrt H) H^{-s} P_n f\|_{L^{p'}} \les |t|^{-s} \|f\|_{L^p}.
$$
Following an argument from \cite{sog}, let $f, g \in L^p$ and write
$$\begin{aligned}
\langle \cos(t \sqrt H) H^{-s} f, g \rangle &= \sum_{n_1, n_2 \in \Z} \langle \cos(t \sqrt H) H^{-s} \tilde P_{n_1} f, \tilde P_{n_2} g \rangle \\
&= \sum_{k=-1}^1 \sum_{n \in \Z} \langle \cos(t \sqrt H) H^{-s} \tilde P_n f, \tilde P_{n+k} g \rangle
\end{aligned}$$
because of the quasi-orthogonality of $\tilde P_n$. Thus all non-zero terms are on three diagonals, namely the main diagonal and the ones above and below. Regarding the main diagonal, for $1<p \leq 2$
$$\begin{aligned}
\bigg|\sum_{n \in \Z} \langle \cos(t \sqrt H) H^{-s} P_n f, P_n g \rangle\bigg| &\les |t|^{-s} \sum_{n \in \Z} \|\tilde P_n f\|_{L^p} \|\tilde P_n g\|_{L^p} \leq |t|^{-s} \|\tilde P_n f\|_{\ell^2_n L^p} \|\tilde P_n g\|_{\ell^2_n L^p} \\
&\leq |t|^{-s} \|\tilde P_n f\|_{L^p \ell^2_n} \|\tilde P_n g\|_{L^p \ell^2_n} = |t|^{-s} \|S_H f\|_{L^p} \|S_H g\|_{L^p} \\
&\les |t|^{-s} \|f\|_{L^p} \|g\|_{L^p},
\end{aligned}$$
where we used Minkowski's inequality to interchange the order of the norms, for $p \leq 2$, and Proposition \ref{sqr} regarding the square function.

The other two diagonals are treated in the same way. This proves the desired inequality.
\end{proof}

\section*{Acknowledgments}
M.B.\;has been supported by the NSF grant DMS--1700293 and by the Simons Collaboration Grant No.\;429698.

M.G.\;is supported by the Simons Collaboration Grant No.\;635369.

M.B.~would like to thank Gang Zhou for the useful conversations on this topic. We would also like to thank the anonymous referee for another paper, who brought Peral's estimates to our attention.


\begin{thebibliography}{ABCD}
\bibitem[Agm]{Ag} Agmon, S. \emph{Spectral properties of Schr\"odinger operators and scattering theory}, Ann.\ Scuola Norm.\ Sup.\ Pisa Cl.\ Sci. (4) {\bf 2} (1975), no.~2, 151--218.
\bibitem[Bea]{beals} M.\;Beals, \emph{Optimal $L^\infty$ decay for solutions to the wave equation with a potential}, Comm.\ Partial Differential Equations {\bf 19} (1994), no.\;7-8, 1319--1369.
\bibitem[Bec1]{bec} M.\;Beceanu, \emph{New estimates for a time-dependent Schr\"{o}dinger equation}, Duke Math.\;J. {\bf 159} (2011), No.\;3, 417--477.
\bibitem[Bec2]{bec2} M.\;Beceanu, \emph{Structure of wave operators for a scaling-critical class of potentials}, Amer.\;J.\;Math.\; {\bf 136} (2014), no.\;2, 255--308.
\bibitem[Bec3]{bec3} M.\;Beceanu, \emph{Dispersive estimates in $\R^3$ with threshold eigenstates and resonances}, Analysis \& PDE 9-4 (2016), pp.\;813--858.
\bibitem[BeGo1]{becgolschr} M.\;Beceanu, M.\; Goldberg, \emph{Dispersive estimates for the Schr\"odinger equation with scaling-critical potential}, Comm.\;Math.\;Phys. {\bf 314} (2012), no.\;2, 471--481.
\bibitem[BeGo2]{becgol2} M.\;Beceanu, M.\;Goldberg, \emph{Strichartz estimates and maximal operators for the wave equation in $\R^3$}, J.\;Funct.\;Anal. {\bf 266} (2014), no.\;3, 1476--1510.
\bibitem[BeGo3]{becgol3} M.~Beceanu, M.~Goldberg, \emph{Spectral multipliers and wave propagation for Hamiltonians with a scalar potential}, arXiv:2207.02987.
\bibitem[BeGo4]{becgol4} M.~Beceanu, M.~Goldberg, \emph{Spectral multipliers II: elliptic and parabolic operators and Bochner--Riesz means}, arXiv:2308.09606.
\bibitem[BeSc]{becschlag} \emph{Structure formulas for wave operators} (in collaboration with W.~Schlag), American Journal of Mathematics (2020) 142, No.~3, pp.~751--807.
\bibitem[BeL\"o]{bergh} J.~Bergh, J.~L\"ofstr\"om, \emph{Interpolation Spaces. An Introduction}, Springer-Verlag, 1976.
\bibitem[CaZy]{cazy} A.P.~Calder\'{o}n, A.~Zygmund, \emph{On the existence of certain singular integrals}, Acta Math.~88, pp.~85--139.
\bibitem[DaPi]{dancpier} P.~D'Ancona, V.~Pierfelice, \emph{On the wave equation with a large rough potential}, J.\;Funct.\;Anal. {\bf 227} (2005), no.\;1, 30--77.
\bibitem[ErSc]{erdschlag} B.~Erdogan, W.~Schlag, \emph{Dispersive estimates for Schrodinger operators in the presence of a resonance and/or an eigenvalue at zero energy in dimension three: I}, Dynamics of PDE 1 (2004), pp.~359--379.
\bibitem[Gol]{goldberg} M.~Goldberg, \emph{Dispersive estimates for Schr\"{o}dinger operators with measure-valued potentials in $\R^3$}, Indiana Univ.\;Math.\;J. {\bf 61} (2012), 2123--2141.
\bibitem[GoSc]{golsch2} M.~Goldberg, W.~Schlag, \emph{A limiting absorption principle for the three-dimensional Schr\"odinger equation with $L^p$ potentials}, Int.\ Math.\ Res.\ Not. {\bf 2004:75} (2004), 4049-4071.
\bibitem[Hon]{hong} Y.\; Hong, \emph{A spectral multiplier theorem associated with a Schr\"odinger operator}, J.\; Fourier Anal.\; Appl. {\bf 22} (2016), no.\;3, 591--622.
\bibitem[IoJe]{ioje} A.\ D.\ Ionescu, D.\ Jerison, \emph{On the absence of positive eigenvalues of Schr\"odinger operators with rough potentials}, Geom.\;Funct.\;Anal. {\bf 13} (2003), no.\;5, 1029--1081.
\bibitem[KeTa]{keeltao} M.\;Keel, T.\;Tao, \emph{Endpoint Strichartz estimates}, Amer.\;Math.\;J. {\bf 120} (1998), 955--980.
\bibitem[KoTa]{kota} H.\;Koch, D.\;Tataru, \emph{Carleman estimates and absence of embedded eigenvalues}, Comm.\ Math.\ Phys. {\bf 267} (2006), no.\;2, 419--449.
\bibitem[Per]{peral} J.~Peral, \emph{$L^p$ estimates for the wave equation}, Journal of Functional Analysis no.~36, pp.~114--145 (1980). 
\bibitem[SYY]{sikora} A.\;Sikora, L.\;Yan, X.\;Yao, \emph{Spectral multipliers, Bochner--Riesz means and uniform Sobolev inequalities for elliptic operators}, Int.\ Math.\ Res.\ Not. {\bf 2018} (2018), no.\;10, 3070-3121.
\bibitem[Sim]{simon} B.~Simon, \emph{Schr\"{o}dinger semigroups}, Bull.~Amer.~Math.~Soc.~no.~7 (1982), pp.\;447--526.
\bibitem[Sog]{sog} C.~D.~Sogge, \emph{Lectures on Nonlinear Wave Equations}, International Press, 1995.
\bibitem[Ste]{stein} E.~Stein, \emph{Harmonic Analysis}, Princeton University Press, 1993.
\bibitem[Wol]{wolff} T.\;Wolff, \emph{Lectures on Harmonic Analysis}, American Mathematical Society, University Lecture Series vol.\;29, 2003.
\bibitem[Zhe]{zheng} S.\;Zheng, \emph{Spectral multipliers for Schr\"odinger operators}, Illinois J.\; Math. {\bf 54} (2010), no.\;2, 621--647.
\end{thebibliography}
\end{document}